\documentclass[a4paper,11pt]{amsart}

\usepackage{amssymb, amsthm, amsmath}
\usepackage{tikz}

\usepackage{hyperref}

\usepackage{bbm}

\usepackage[hmargin=3.5cm,foot=0.7cm]{geometry}
\usepackage[utf8x]{inputenc}

\usepackage[english]{babel}
\usepackage[abbrev]{amsrefs}

\usepackage{tikz}
\usepackage{tikz-cd}
\usetikzlibrary{matrix,arrows}
\usepackage{pgfplots}
\pgfplotsset{width=7cm,compat=1.8}
\usepgfplotslibrary{polar}

\usepackage{enumitem}
\usepackage{verbatim}

\newtheorem{theorem}{Theorem}[section]
\newtheorem*{theorem*}{Theorem}
\newtheorem{definition}[theorem]{Definition}

\theoremstyle{plain}
\newtheorem{corollary}[theorem]{Corollary}
\newtheorem*{corollary*}{Corollary}
\newtheorem{lemma}[theorem]{Lemma}
\newtheorem{proposition}[theorem]{Proposition}

\newcounter{mt}

\newtheorem{MainTheorem}[mt]{Theorem}
\newtheorem{MainCorollary}[mt]{Corollary}

\newtheorem*{lemma*}{Lemma}
\newtheorem*{question*}{Question}

\theoremstyle{definition}
\newtheorem{example}[theorem]{Example}

\makeatletter
\newcommand{\tpitchfork}{%
  \vbox{
    \baselineskip\z@skip
    \lineskip-.52ex
    \lineskiplimit\maxdimen
    \m@th
    \ialign{##\crcr\hidewidth\smash{$-$}\hidewidth\crcr$\pitchfork$\crcr}
  }%
}
\makeatother

\BibSpec{article}{%
  +{}{\PrintAuthors}  		{author}
  +{,}{ \textit}     		{title}
  +{,}{ }             		{journal}
  +{}{ \textbf}       		{volume}
  +{}{ \parenthesize} 		{date}
  +{}{, no. } 		{number}
  +{,}{ }      	      		{conference}
  +{,}{ }      	      		{book}
  +{,}{ }            		{pages}
  +{,}{ }            	 	{note}
  +{,}{ }            	 	{status}
  +{,}{  \texttt } {eprint}
  +{.}{}              {transition}
}

\BibSpec{book}{%
  +{}{\PrintAuthors}  {author}
  +{,}{ \textit}      {title}
  +{,}{ }             {publisher}
  +{,}{ }             {place}
  +{,}{ }             {date}
  +{.}{}              {transition}
}

\BibSpec{incollection}{%
    +{}  {\PrintAuthors}                {author}
    +{,} { \textit}                     {title}
    +{.} { }                            {part}
    +{:} { \textit}                     {subtitle}
    +{,} { \PrintContributions}         {contribution}
    +{,} { \PrintConference}            {conference}
    +{,} { }                            {booktitle}
    +{,} { pp.~}                        {pages}
    +{,}{ }       		{series}
    +{}{ \textbf}       		{volume}
    +{,} { }                            {publisher}
    +{,} { }                 {date}
    +{,} { }                            {status}
    +{,} { \PrintDOI}                   {doi}
    +{,} { available at \eprint}        {eprint}
    +{}  { \parenthesize}               {language}
    +{}  { \PrintTranslation}           {translation}
    +{;} { \PrintReprint}               {reprint}
    +{.} { }                            {note}
    +{.} {}                             {transition}
}

\newcommand{\NN}{\mathbb{N}}

\newcommand{\RR}{\mathbb{R}}

\newcommand{\Conv}{\mathrm{Conv}}
\newcommand{\ConvO}{\mathrm{Conv}_{(0)}}
\newcommand{\LC}{\mathrm{LC}}

\DeclareMathOperator{\dom}{dom}

\DeclareMathOperator{\epi}{epi}

\newcommand{\GW}{\mathrm{GW}}

\newcommand{\MeasC}{\mathcal{M}_c}
\newcommand{\Meas}{\mathcal{M}}
\DeclareMathOperator{\cent}{cent}

\newcommand{\DistribC}{\mathcal{D}_c'}
\newcommand{\Distrib}{\mathcal{D}'}

\newcommand{\convexbodies}{\mathcal{K}^n}
\newcommand{\unitsurfn}{{\mathbb{S}^{n-1}}}

\DeclareMathOperator{\GL}{GL}

\DeclareMathOperator{\SO}{SO}

\DeclareMathOperator{\id}{id}
\DeclareMathOperator{\supp}{supp}

\DeclareMathOperator{\interior}{int}

\DeclareMathOperator{\Diffbody}{D}
\newcommand{\MongAmp}[2]{\mathrm{MA}(#1;#2)}

\title{Equivariant Endomorphisms of Convex Functions}
\author{Georg C. Hofst\"atter}
\author{Jonas Knoerr}
\email{georg.hofstaetter@aon.at}
\email{jonas.knoerr@tuwien.ac.at}
\address{School of Mathematical Sciences, Tel Aviv University, Tel Aviv 69978, Israel}
\address{Institut für Diskrete Mathematik und Geometrie, TU Wien, 1040 Wien, Austria}

\thanks{GH was supported by the European Research Council (ERC)
under the European Union’s Horizon 2020 research and innovation programme (grant agreement No 770127) and by the Austrian Science Fund
(FWF), Project number: P31448-N35.}

\begin{document}

\begin{abstract}
 Characterizations of all continuous, additive and $\GL(n)$-equivariant endomorphisms of the space of convex functions on a Euclidean space $\RR^n$, of the subspace of convex functions that are finite in a neighborhood of the origin, and of finite convex functions are established. Moreover, all continuous, additive, monotone endomorphisms of the same spaces, which are equivariant with respect to rotations and dilations, are characterized. Finally, all continuous, additive endomorphisms of the space of finite convex functions of one variable are characterized.
\end{abstract}

\maketitle

\section{Introduction}

The study of additive functions on the space $\convexbodies$ of convex bodies, that is, the space of all non-empty, convex and compact subsets of $\RR^n$, has a long history. Here, a map $\mu:\convexbodies\rightarrow (G,+)$ into an Abelian semi-group $(G,+)$ is called additive if
\begin{align*}
	\mu(K+L)=\mu(K)+\mu(L),\quad\text{for all }K,L\in\convexbodies,
\end{align*}
where $K+L=\{x+y: x\in K,y\in L\}$ is the Minkowski sum of $K,L\in\convexbodies$. Functions of this type include the mean width (with real values) and the Steiner point (with values in $\RR^n$), which are both essentially uniquely characterized by this property, continuity (with respect to the Hausdorff metric), and certain compatibility properties with respect to rigid motions, as was shown in \cites{Shephard1968, Schneider1971}.

\medskip

Initiated by Schneider in a seminal paper \cite{Schneider1974c}, additive endomorphisms of $\convexbodies$, that is, additive maps on $\convexbodies$ with values in $\convexbodies$, have been the focus of intense research in the past 50 years \cites{Abardia2018, Dorrek2017b, Hofstaetter2021, Kiderlen2006,Ludwig2005, Schneider1974b, Schuster2007}.
The first result in this direction is the following classification of the \emph{difference body} $\Diffbody K:=K+(-K)$ for $K\in\convexbodies$.

\begin{theorem}[Schneider~\cite{Schneider1974c}]\label{thm:SchneiderGLEquiv}
A map $\Phi: \convexbodies \rightarrow \convexbodies$ is continuous, Minkowski addi\-ti\-ve, $\GL(n)$-equivariant, that is, $\Phi(\eta K) = \eta \Phi(K)$ for $K \in \convexbodies$ and $\eta \in \GL(n)$, and translation-invariant if and only if $\Phi = c\Diffbody$ for some $c \geq 0$. 
\end{theorem}
If we replace $\GL(n)$- by $\SO(n)$-equivariance in the conditions of the theorem, we obtain the notion of \emph{Minkowski endomorphisms}, introduced by Schneider, who also gave a characterization of this class of maps in dimension $n=2$ (see \cite{Schneider1974}). For $n\ge 3$, no complete classification of all Minkowski endomorphisms is known. However, for \emph{monotone} Minkowski endomorphisms, that is, $\Phi:\convexbodies\rightarrow\convexbodies$ satisfying additionally
\begin{align*}
	\Phi(K)\subset \Phi(L)\quad\text{for all }K,L\in\convexbodies\text{ such that }K\subset L,
\end{align*}
a complete characterization was obtained by Kiderlen \cite{Kiderlen2006}. For the statement of this result, let $\unitsurfn\subset\RR^n$ denote the unit sphere and $\Meas^+(\unitsurfn)$ the space of finite, non-negative Borel measures on $\unitsurfn$, and choose $\vartheta_u \in \SO(n)$, for every $u \in \unitsurfn$, such that $\vartheta_u \bar e = u$, where $\bar e \in \unitsurfn$ is in the stabilizer of $\SO(n-1)$ (a pole of $\unitsurfn$).
\begin{theorem}[Kiderlen~\cite{Kiderlen2006}]\label{thm:KiderlenMinkEndos}
A map $\Phi: \convexbodies \rightarrow \convexbodies$ is a monotone Minkowski endomorphism if and only if there exists an $\SO(n-1)$-invariant $\mu \in \Meas^+(\unitsurfn)$ with $\cent \mu = 0$ such that
\begin{align} \label{eq:kiderlenCharMonMinkEndos}
h(\Phi K,u) = \int_{\unitsurfn} h(K,\vartheta_u v) d\mu(v), \quad u \in \unitsurfn,
\end{align}
for every $K \in \convexbodies$. Moreover, the measure $\mu$ is uniquely determined by $\Phi$.
\end{theorem}
Here, $\mu(\unitsurfn) \cdot \cent\mu = \int_\unitsurfn y d\mu(y)$ denotes the center of mass of $\mu\in\Meas^+(\unitsurfn)$ and $h(K,y):=\sup_{x\in K}\langle y,x\rangle$ for $y\in\RR^n$ denotes the support function of $K\in\convexbodies$.

\medskip

Kiderlen derived Theorem~\ref{thm:KiderlenMinkEndos} from an interpretation of Minkowski endomorphisms as certain distributions on the sphere, an approach previously used by Goodey--Weil~\cite{Goodey1984}. The monotonicity of the endomorphisms then implies that these functionals admit a representation by non-negative measures, as in Theorem~\ref{thm:KiderlenMinkEndos}. This reduces the problem to showing that the right-hand side of \eqref{eq:kiderlenCharMonMinkEndos} defines a support function for every convex body $K$, which is non-trivial. Dorrek~\cite{Dorrek2017b} refined this construction and gave (first) examples of non-monotone Minkowski endomorphisms.

\medskip

In recent years, many classical notions from convex geometry have been generalized to function spaces, for example to log-concave functions \cites{Colesanti2017d, Colesanti2006, Colesanti2013, Rotem2012, Rotem2021}. There now exists a substantial body of research on valuations on function spaces \cites{Alesker2019,Cavallina2015,Colesanti2017c,Colesanti2018,Colesanti2019, Colesanti2017,Colesanti2019b,Colesanti2017b,Colesanti2020,Colesanti2021,Colesanti2021b,Colesanti2022,Knoerr2020a,Knoerr2020b,Knoerr2021,Mussnig2019}, functional versions of geometric inequalities \cites{Artstein2004,Barthe2014, Fradelizi2007, Haddad2020, Hofstaetter2021, Kolesnikov2020, Lehec2009}, and classifications of many natural operations on functions \cites{Artstein2009,Artstein2010, Artstein2011, Klartag2005, Milman2013,Milman2013b}, including the study of additive maps on convex functions under bijectivity conditions.

In \cite{Hofstaetter2021}, the notion of \emph{Asplund endomorphism} was introduced, extending the ideas of Schneider to endomorphism of spaces of log-concave functions, that is, of functions $\varphi:\RR\rightarrow[0,\infty)$ such that $\varphi=e^{-f}$ with convex $f:\RR^n\rightarrow(-\infty,\infty]$. More precisely, the space $\LC_c(\RR^n)$ of all proper log-concave
functions which are upper semi-continuous and coercive was considered. Here, a function $\varphi:\RR^n\rightarrow[0,\infty)$ is called \emph{proper}
if it is not identically $0$ and it is called \emph{coercive} if $\lim_{\|x\|\rightarrow\infty} \varphi(x) = 0$. The notion of Minkowski addition naturally extends to this space: If $\varphi,\psi\in \LC_c(\RR^n)$ are given, their \emph{Asplund sum} (or sup-convolution) is defined by
\begin{align*}
	(\varphi\star \psi)(x)=\sup_{x_1+x_2=x}\varphi(x_1)\psi(x_2), \quad x \in \RR^n.
\end{align*}
It is easy to see that $\mathbbm{1}_K\star\mathbbm{1}_L=\mathbbm{1}_{K+L}$ for indicator functions of $K,L\in\convexbodies$. For general log-concave functions, their Asplund sum may attain the value $+\infty$, however, $\LC_c(\RR^n)$ is closed under this operation (see, e.g., \cite{Hofstaetter2021}*{Lem.~2.3}). Using these notions, a map $\Psi:\LC_c(\RR^n)\rightarrow\LC_c(\RR^n)$ is called an Asplund endomorphism if it is 
\begin{enumerate}
	\item continuous with respect to the topology induced by hypo-convergence,
	\item Asplund additive,
	\item translation-invariant, that is, $\Psi(\varphi(\cdot+x))=\Psi(\varphi)$,
	\item $\SO(n)$-equivariant, that is, $\Psi(\varphi\circ\eta)[x]=\Psi(\varphi)[\eta x]$,%
\end{enumerate}
for all $\varphi\in \LC_c(\RR^n)$, $x\in\RR^n$ and $\eta \in \SO(n)$. Asplund endomorphisms can be constructed using the support function $h(\varphi,\cdot):=\mathcal{L}(-\log \varphi)$ of $\varphi\in\LC_c(\RR^n)$ (following \cite{Artstein2010}), where $\mathcal{L}$ denotes the classical Legendre transform. Basic properties of the Legendre transform imply that an Asplund endomorphism defines an additive map on the level of support functions, which can be used to construct a large number of non-trivial examples. In fact, every monotone Minkowski endomorphism extends to an Asplund endomorphism, as was shown in \cite{Hofstaetter2021}. To state the result, we choose $\vartheta_x \in \SO(n)$, for every $x \neq 0$, such that $\vartheta_x \bar e = \frac{x}{\|x\|}$. Then $\vartheta_x$ is defined up to right-multiplication by $\SO(n-1)$.

\begin{theorem}[\cite{Hofstaetter2021}*{Thm.~3}] \label{thm:HSKlasseAsplEndos}
Each $\SO(n-1)$-invariant $\mu \in \Meas^+(\unitsurfn)$ with $\cent \mu = 0$ induces a monotone Asplund endomorphism $\Psi_{\mu}$ by
\begin{align} \label{thm3introform}
h(\Psi_{\mu} \varphi,x) = \int_{\unitsurfn} h(\varphi,\|x\|\vartheta_x v) d\mu(v), \quad x \in \RR^n \setminus \{0\},
\end{align}
for $\varphi  \in \LC_c(\RR^n)$. Moreover, the measure $\mu$ is uniquely determined by $\Psi_{\mu}$.
\end{theorem}
This result was then used in \cite{Hofstaetter2021} to obtain new functional inequalities for these endomorphisms. 
While the focus on log-concave functions is natural from this perspective, the goal of the present article is a more systematic study of additive maps of this type with the aim to provide a classification under certain equivariance properties with respect to the standard representation of $\GL(n)$ or subgroups. 
We will thus work directly with (convex) support functions instead of the corresponding log-concave functions. In this setting, translation-invariance corresponds to invariance with respect to addition of linear functions, and we will call maps with this property \emph{dually translation-invariant}.

To be more precise, we will be concerned with additive maps defined on subspaces of the space $\Conv(\RR^n)$ of all lower semi-continuous, convex functions $f:\RR^n\rightarrow(-\infty,+\infty]$ that are \emph{proper}, that is, not identically $+\infty$. We will mostly be interested in the subspace $\Conv(\RR^n,\RR)$ of finite-valued convex functions, that is, convex functions $\varphi:\RR^n\rightarrow\RR$, and the space
\begin{align*}
	\ConvO(\RR^n):=\{f\in\Conv(\RR^n): f<+\infty\text{ on a neighborhood of }0\in\RR^n\}.
\end{align*}
Our interest in the latter space stems from the fact that support functions of elements in $\LC_c(\RR^n)$ belong to this class, which establishes a $\GL(n)$-equivariant bijection between the two spaces that intertwines Asplund sum and pointwise addition.%

Obviously, $\Conv(\RR^n,\RR) \subset \ConvO(\RR^n) \subset \Conv(\RR^n)$. Note that $\Conv(\RR^n,\RR)$ and $\ConvO(\RR^n)$ are both closed under pointwise addition of functions, while the sum of two elements in $\Conv(\RR^n)$ may be identical to $+\infty$ and thus not in $\Conv(\RR^n)$. We will therefore call a map $\Psi:C\rightarrow\Conv(\RR^n)$, defined on $C\subset \Conv(\RR^n)$, \emph{additive} if
\begin{align*}
	\Psi(g+h)=\Psi(g)+\Psi(h),\quad\text{for all } g,h\in C\text{ such that }g+h\in C.
\end{align*}
In particular, we require that $\Psi(g)+\Psi(h) \in \Conv(\RR^n)$ whenever $g+h\in C$. 

\medskip

Before discussing the main results of this article, we want to note that $\Conv(\RR^n,\RR)$ is dense in both $\ConvO(\RR^n)$ and $\Conv(\RR^n)$  if we equip these spaces with the topology induced by epi-convergence. In particular, any continuous and additive map from any one of these spaces to $\Conv(\RR^n)$ is uniquely determined by its restriction to finite-valued convex functions. Consequently, any classification problem of additive maps on these spaces can be split into two parts: 
\begin{enumerate}
	\item classify the corresponding maps on finite-valued convex functions,
	\item determine which maps extend to the desired class of functions by continuity. %
\end{enumerate}

Following this strategy, we prove our first main result, which is an analogue of Theorem~\ref{thm:SchneiderGLEquiv} for the space $\ConvO(\RR^n)$. To state it, let $\MeasC^+(\RR)$ denote the space of non-negative Borel measures on $\RR$ with compact support and set $\RR^\times = \RR \setminus \{0\}$. %
\begin{MainTheorem}\label{mthm:CharGLEquiv}
 A map $\Psi: \ConvO(\RR^n) \rightarrow \ConvO(\RR^n)$ is continuous, additive and $\GL(n)$-equi\-variant if and only if there exists $\nu\in\MeasC^+(\RR)$ with $\int_{\RR^\times}|s|^{-1} d\nu(s) < \infty$ and $c \in \RR$ such that
 \begin{align}\label{eq:MainThmGLEquiv}
  \Psi(f)[x] = cf(0) + \int_{\RR^\times} \frac{f(sx)-f(0)}{|s|^2} d\nu(s), \quad x \in \RR^n,
 \end{align}
 for every $f \in \ConvO(\RR^n)$. Moreover, the map $\Psi$ defined by \eqref{eq:MainThmGLEquiv} is
 \begin{itemize}
  \item monotone if and only if $\int_{\RR^\times}|s|^{-2}d\nu(s) \leq c < \infty$.
  \item dually translation-invariant if and only if $\int_{\RR^\times} s^{-1} d\nu(s) = 0$.
 \end{itemize}
\end{MainTheorem}
Note that $c=\Psi(1)[0]$ is uniquely determined by $\Psi$, the same holds for $\nu$ if we require $\nu(\{0\})=0$.

The idea of the proof is similar to Kiderlen's approach. Given an additive and continuous map $\Psi:\ConvO(\RR^n)\rightarrow\ConvO(\RR^n)$, any such endomorphism restricts to an endomorphism of $\Conv(\RR^n,\RR)$ by $\GL(n)$-equivariance. By evaluating the function $\Psi(f)$ for $f \in \Conv(\RR^n,\RR)$ at points $x \in \RR^n$, we obtain continuous, additive, real-valued functionals on $\Conv(\RR^n,\RR)$ -- in other words, distributions. We then analyze which families of distributions are compatible with the equivariance property and, among these, which yield mappings into $\Conv(\RR^n, \RR)$. %

Thus, the proof of Theorem \ref{mthm:CharGLEquiv} actually provides a classification of all such equivariant endomorphisms on the space $\Conv(\RR^n,\RR)$, but the relevant maps extend to endomorphisms on the larger space $\ConvO(\RR^n)$. In particular, Theorem~\ref{mthm:CharGLEquiv} looks the same if $\ConvO(\RR^n)$ is replaced by $\Conv(\RR^n,\RR)$. We may even consider equi\-va\-riant maps $\Conv(\RR^n,\RR)\rightarrow\ConvO(\RR^n)$ without changing the class of functionals. The situation is completely different, however, if we replace $\ConvO(\RR^n)$ with $\Conv(\RR^n)$:

\begin{MainTheorem} \label{mcor:CharGLEquivWholeConv}
 A map $\Psi: \Conv(\RR^n) \rightarrow \Conv(\RR^n)$ is continuous, additive, and $\GL(n)$-equivariant if and only if either $\Psi \equiv 0$ or $\Psi \equiv \mathbbm{1}_{\{0\}}^\infty$ or there exists $\lambda > 0$ and $\mu \in \RR^\times$ such that
 \begin{align}\label{eq:MainCorGLEquivWholeConv}
  \Psi(f)[x] = \lambda f(\mu x), \quad x \in \RR^n,
 \end{align}
 for every $f \in \Conv(\RR^n)$.
\end{MainTheorem}
In particular, the difference body $\Diffbody=\id+(-\id)$ does not possess an extension to $\Conv(\RR^n)$, but the Minkowski additive endomorphisms $\pm \id$ do. The space of additive maps on $\Conv(\RR^n)$ therefore is not closed under addition, which is is hardly surprising, as $\Conv(\RR^n)$ is not closed under pointwise addition of functions itself. Moreover, by a result in \cite{Artstein2010}, the non-constant endomorphisms in Theorem~\ref{mcor:CharGLEquivWholeConv} are exactly the bijective, $\GL(n)$-equivariant endomorphisms of $\Conv(\RR^n)$.

\medskip

We now turn to questions in the spirit of Kiderlen's Theorem~\ref{thm:KiderlenMinkEndos} about $\SO(n)$-equivariant endomorphisms $\Psi$ that are monotone, that is, $\Psi(f)[x] \leq \Psi(g)[x]$, $x \in \RR^n$, holds for every two functions $f$ and $g$ in the domain of $\Psi$, satisfying $f(y) \leq g(y)$ for all $y \in \RR^n$. Similar to Kiderlen's approach, it is possible to give representation formulas for monotone and additive maps on $\Conv(\RR^n,\RR)$ in terms of certain families of measures (derived from the associated distributions, as in the $\GL(n)$-equivariant case discussed before). However, it is in general very difficult to decide whether a given family of measures actually defines an endomorphisms, as the class of possible families seems to be too big to handle. One major obstacle is that the different measures in the family might a-priori be unrelated (stemming from the fact that there are infinitely many $\SO(n)$-orbits in $\RR^n$). We therefore need some additional (equivariance) condition that relates the different $\SO(n)$-orbits.

Note that the maps defined in Theorem~\ref{thm:HSKlasseAsplEndos}, when considered as endomorphisms of $\ConvO(\RR^n)$, that is, given for $\SO(n-1)$-invariant $\mu\in\Meas^+(\unitsurfn)$ by 
\begin{align*} 
\Psi_\mu(f)[x] = \int_{\unitsurfn} f(\|x\|\vartheta_x v) d\mu(v), \quad x \in \RR^n \setminus \{0\}, %
\end{align*}%
satisfy $\Psi_\mu (f)[t\cdot x]=\Psi_\mu (f(t\cdot))[x]$ for all $t>0$, $x\in\RR^n$, $f\in\ConvO(\RR^n)$. Hence, they are equivariant with respect to dilations of $\RR^n$, which we call \emph{radially equivariant} in the following. This reduces the number of orbits of the joint representation to two, which allows us to prove our next main result, an analogue of Kiderlen's~Theorem~\ref{thm:KiderlenMinkEndos}.
\begin{MainTheorem}\label{mthm:CharMonRadEquiv}
 A map $\Psi: \ConvO(\RR^n) \rightarrow \ConvO(\RR^n)$ is continuous, additive, monotone, as well as radially and $\SO(n)$-equivariant if and only if there exists a (necessarily unique) $\SO(n-1)$-invariant measure $\mu\in\MeasC^+(\RR^n)$ such that
 \begin{align}\label{eq:MainThmRadiallyEquiv}
  \Psi(f)[x] = \int_{\RR^n} f(\|x\|\vartheta_x y) d\mu(y), \quad x \in \RR^n \backslash\{0\},
 \end{align}
 and $\Psi(f)[0] = \liminf_{\|x\|\to 0} \Psi(f)[x]=f(0)\mu(\RR^n)$, for every $f \in \ConvO(\RR^n)$.\\
 Moreover, $\Psi$ is dually translation-invariant if and only if $\int_{\RR^n} y d\mu(y) = 0$.
\end{MainTheorem}
As a corollary, we obtain a classification of the endomorphisms in Theorem~\ref{thm:HSKlasseAsplEndos} as precisely those maps in this class that act as multiples of the identity on radially symmetric convex functions (see Corollary~\ref{cor:CharHSKlasseAsplEndos}).

Similar to the $\GL(n)$-equivariant case, Theorem~\ref{mthm:CharMonRadEquiv} does not change when we replace $\ConvO(\RR^n)$ by $\Conv(\RR^n, \RR)$, but the situation is again completely different for $\Conv(\RR^n)$.
\begin{MainCorollary} \label{mcor:CharRadiallyEquivWholeConv}
 A map $\Psi: \Conv(\RR^n) \rightarrow \Conv(\RR^n)$ is continuous, additive, monotone, as well as radially and $\SO(n)$-equivariant if and only if $\Psi \equiv 0$ or $\Psi \equiv \mathbbm{1}_{\{0\}}^\infty$ or there exists $\lambda > 0$ and $\mu \in \RR^\times$ such that
 \begin{align}\label{eq:MainCorRadEquivWholeConv}
  \Psi(f)[x] = \lambda f(\mu x), \quad x \in \RR^n,
 \end{align}
 for every $f \in \Conv(\RR^n)$.
\end{MainCorollary}

Up to this point, all of our results exploit the fact that the equivariance properties reduce the classification problem to one distribution (or measure) per orbit. In general, the situation is of course much more complicated. In particular, it is usually not the case that an endomorphism of $\Conv(\RR^n,\RR)$ extends to an endomorphism of $\ConvO(\RR^n)$ and we provide examples of such mappings in Section~\ref{subsection:examples}.

\medskip

If we restrict ourselves to endomorphisms of $\Conv(\RR,\RR)$, the situation is much simpler, and we are able to provide a full characterization of all endomorphisms without any additional assumptions concerning monotonicity or equivariance. To state the result, recall that the derivative of  a distribution $u\in\mathcal{D}'(\RR)$ is defined by the relation $(\partial^k u)(\phi):=(-1)^k u(\phi^{(k)})$ for $\phi\in C^\infty_c(\RR)$, $k \in \NN$. If $g:\RR\rightarrow\RR$ is continuous, we may in particular  consider the distribution $\partial^2 g$ defined by $(\partial^2 g)(\phi)=\int_{\RR}\phi''(y)g(y)dy$ for $\phi\in C^\infty_c(\RR)$. In the following theorem, all partial derivatives are understood in this distributional sense.
\begin{MainTheorem}\label{mthm:Char1Dim}
 A map $\Psi: \Conv(\RR,\RR) \rightarrow \Conv(\RR,\RR)$ is continuous and additive if and only if there exists $\psi \in C(\RR^2)$ such that
 \begin{align}\label{eq:1DCharRep}
 \Psi(f)[x] =(\partial^2_y \psi(x,\cdot))(f), \quad x \in \RR,
 \end{align}
 for every $f \in \Conv(\RR,\RR)\cap C^2(\RR)$, where $\psi$ has the following properties:
 \begin{enumerate}
  \item \label{propCharDim1_CondConv} $\psi(\cdot,y)$ is convex for every $y \in \RR$.
  \item \label{propCharDim1_CondSupps}   For every compact subset $A\subset\RR$ there exists $R=R(A)>0$ such that
    \begin{enumerate}
        \item \label{propCharDim1_CondSuppsA}$\supp\partial_x^2\psi(\cdot,y)\cap A=\emptyset$ for all $y\in\RR\setminus[-R,R]$;
        \item \label{propCharDim1_CondSuppsB}$\supp\partial_y^2\psi(x,\cdot)\subseteq [-R,R]$ for all $x\in A$.
    \end{enumerate}
 \end{enumerate}
 Moreover, $\Psi$ is monotone if and only if $\psi(x,\cdot)$ is convex for every $x\in\RR$.
\end{MainTheorem}

Note that property~\eqref{propCharDim1_CondSuppsB} ensures that the distribution given by $\partial^2_y \psi(x,\cdot)$ has compact support and hence \eqref{eq:1DCharRep} is well-defined. The function $\psi$ is not uniquely determined by $\Psi$, but the possible modification can be described completely (see Corollary~\ref{cor:Uniqueness1Dpsi}). Moreover, the proof shows that a candidate for $\psi$ can be obtained from the underlying endomorphism $\Psi$ by plugging in a suitable convex function; the converse is slightly more involved. We provide an explicit representation of $f\mapsto \Psi(f)[x]$ using results on dually epi-translation invariant valuations on $\Conv(\RR, \RR)$ obtained by Colesanti--Ludwig--Mussnig \cite{Colesanti2019b}. However, these formulas only hold locally around $x\in\RR$. In particular, it is not easy to see under which conditions these endomorphisms extend to $\ConvO(\RR)$.

\medskip

As a closing remark we revisit some of the examples considered in Section \ref{section:GW} and reinterpret them in terms of Theorem \ref{mthm:Char1Dim}. This also provides some non-trivial examples for functions $\psi\in C(\RR^2)$ satisfying the conditions above.

\medskip
This article is structured as follows: In Section~\ref{sec:background}, we recall the necessary general background on convex functions. In Section~\ref{section:GW} we construct Goodey--Weil distributions, prove auxiliary results for endomorphisms and give examples. In the last three sections, finally, we prove the main results.

\section{Background on Convex Functions}
\label{sec:background}

In this section we recall additional basic notions and results about convex functions. As general references, we recommend the monographs by Rockafellar \cite{Rockafellar1970}, Rockafellar and Wets \cite{Rockafellar1998}, Schneider \cite{Schneider2014}, and Artstein-Avidan, Giannopoulos and Milman \cite{Artstein2015}. 

First recall that to every convex function $f: \RR^n \to (-\infty, \infty]$ we can associate its \emph{domain}, $\dom f = \{x\in\RR^n: f(x)<\infty\}$, and its \emph{epigraph}, $\epi f = \{(x,\xi) \in \RR^n \times \RR: f(x) \geq \xi\}$, which are convex sets. Note that, if $f \in \Conv(\RR^n)$, $\dom f$ is non-empty and $\epi f$ is closed and non-empty.

The spaces $\Conv(\RR^n), \ConvO(\RR^n)$ and $\Conv(\RR^n, \RR)$ defined in the introduction are equipped with the topology induced by epi-convergence (also called $\Gamma$-convergence), which corresponds to convergence of the epi-graphs in the Painlev\'e-Kuratowski sense (cf. \cite{Rockafellar1998}*{Sec. 7B}). A sequence $f_j$ of functions in $\Conv(\RR^n)$ (and therefore also in its subspaces $\ConvO(\RR^n)$ and $\Conv(\RR^n, \RR)$) is called \emph{epi-convergent} to $f: \RR^n \rightarrow (-\infty,\infty]$ if for all $x \in \RR^n$ the following two conditions hold:
\begin{enumerate}
\item\label{defEpiConvInfBed} $f(x) \leq \liminf_{j \rightarrow \infty} f_j(x_j)$ for every sequence $x_j$ that converges to $x$.
\item\label{defEpiConvConvSeq} There exists a sequence $x_j$ converging to $x$ such that $f(x) = \lim_{j \rightarrow \infty} f_j(x_j)$.
\end{enumerate}

The limit function $f$ is then necessarily convex and lower semi-continuous, but might be equal to $+\infty$. Let us also note that epi-convergence extends Hausdorff convergence in the sense that $\mathbbm{1}_{K_j}^\infty$ epi-converges to $\mathbbm{1}_K^\infty$ whenever $K_j \in \convexbodies$ converges to $K \in \convexbodies$ in the Hausdorff distance. Here, $\mathbbm{1}_K^\infty$ denotes the convex indicator function of $K \in \convexbodies$, defined by
\begin{align*}
 \mathbbm{1}_K^\infty(x) = \begin{cases}
                            0, & x \in K,\\
                            +\infty, & \text{ else}.
                           \end{cases}
\end{align*}
The following lemma gives a very useful condition for epi-convergence when the limit function (or a candidate for it) is known to be convex and finite in an open set.

\begin{lemma}[\cite{Rockafellar1998}*{Thm.~7.17}] \label{lem:epiConvEquiv}
If $f, f_j \in \Conv(\RR^n)$ and $\interior\dom f$ is non-empty, then the following statements are equivalent to $(f_j)_j$ being epi-convergent to~$f$:
\begin{enumerate}
\item \label{lem:epiConvEquivPWonDense} There exists a dense set $D \subseteq \RR^n$ such that $f_j(x) \rightarrow f(x)$ for every $x \in D$.
\item \label{lem:epiConvEquivUnifComp} The sequence $f_j$ converges uniformly to $f$ on every compact subset of $\RR^n$ that does not contain a boundary point of $\dom f$.
\end{enumerate}
\end{lemma}

\noindent
An easy application of Lemma~\ref{lem:epiConvEquiv} is the proof of continuity of the addition of functions of $\ConvO(\RR^n)$ and of the standard $\GL(n)$-representation on $\RR^n$, given for $\eta \in \GL(n)$ by $(\eta \cdot f)(x) = f(\eta^{-1}x)$, $x \in \RR^n$. See, e.g., \cite{Knoerr2020a}*{Lem.~4.7} for a proof of the first statement.

\begin{lemma}
	\label{lem:contAddition}
	The maps
	
	\begin{minipage}{0.35\textwidth}
	\begin{align*}
		\ConvO(\RR^n)\times\ConvO(\RR^n)&\rightarrow \ConvO(\RR^n)\\
		(f_1,f_2)&\mapsto f_1+f_2
	\end{align*}
	\smallskip
	\end{minipage}
	$\,\,$
	and 
	$\,\,$
	\begin{minipage}{0.35\textwidth}
    \begin{align*}
        \ConvO(\RR^n) &\rightarrow \ConvO(\RR^n)\\
        f &\mapsto \eta \cdot f
	\end{align*}
	\smallskip
	\end{minipage}
	
	\noindent are continuous for every $\eta \in \GL(n)$.
\end{lemma}
\noindent
We will need the following two auxiliary statements about convex functions.%
\begin{lemma}
	\label{lem:unifLowerBound}
	Let $f_j,f\in\Conv(\RR^n)$, for $j\in\NN$, and $A \subseteq \RR^n$ be compact. If $(f_j)_j$ epi-converges to $f$, then there exists an affine function $g:\RR^n \rightarrow \RR$ such that
	\begin{align*}
		f_j(x)\geq g(x), \quad f(x)\geq g(x)\quad\forall x \in A, j\in\NN.
	\end{align*}
\end{lemma}
\begin{proof}
 Let $g'$ be any affine function that bounds $f$ from below, which exists due to the convexity of $f$. Then their epi-graphs satisfy $\epi g' \supseteq \epi f$. As epi-convergence is equivalent to the convergence of the epigraphs, \cite{Rockafellar1998}*{Thm.~4.10(b)} implies that for every $\varepsilon > 0$ and every set $A' = A \times [-R,R]$, $R>0$, there exists $j_0 \in \NN$, such that
 \begin{align*}
  \epi f_j \cap A' \subseteq \epi (g'-\varepsilon), \quad j \geq j_0
 \end{align*}
 In particular, by choosing $R>0$ such that $g'(x)-\varepsilon> -R$ for every $x \in A$, we deduce that $f_j(x) \geq g'(x)-\varepsilon$, $x \in A$ and $j \geq j_0$. Indeed, if $f_j(x) < g'(x) - \varepsilon$ for some $x \in A$ and $j\geq j_0$, then $c:=\max\{f_j(x),-R\} < g'(x)-\varepsilon$. Thus $(x,c) \not \in \epi (g'-\varepsilon)$, but $(x,c) \in \epi f_j \cap A'$, a contradiction.
 
 Finally, note that by semi-continuity, every $f_j - g'$ is bounded from below on the compact set $A$. Letting $\varepsilon > 0$ be arbitrary and choosing $C \in \RR$ to be a common lower bound for all $j \leq j_0$, the claim follows by setting $g := g' - \varepsilon + \min\{C,0\}$.
\end{proof}

In order to state the next proposition, denote by $B_\varepsilon(x)$ the ball of radius $\varepsilon$ centered at $x \in \RR^n$.

\begin{proposition}[\cite{Rockafellar1998}*{Ex.~9.14}]
	\label{prop:uniformLipschitz}
	Let $U\subset \RR^n$ be a convex open subset and $f:U\rightarrow\RR$ a convex function. If $X\subset U$ is a set with $X+ B_\varepsilon(0)\subset U$ and $f$ is bounded on $X+ B_\varepsilon(0)$, then $f$ is Lipschitz continuous on $X$ with Lipschitz constant $\frac{2}{\varepsilon}\sup_{x\in X+B_\varepsilon}|f(x)|$.
\end{proposition}

\section{Endomorphisms and Goodey--Weil Distributions}
\label{section:GW}

In this section, we will construct a family of Goodey--Weil distributions for every endomorphism of $\Conv(\RR^n,\RR)$, prove some basic results for endomorphisms and give some examples.

\medskip

Before describing the constructions, let us recall some basic notation for distributions. As a general reference on distributions we recommend the book by H\"ormander \cite{Hoermander2003}. We denote by $C_c^\infty(\RR^n)$ the space of all smooth functions on $\RR^n$ with compact support (test functions) and denote by $\Distrib(\RR^n, F)$ the space of all distributions on $\RR^n$ with values in a complete topological vector space $F$, that is, all continuous linear maps $C_c^\infty(\RR^n) \rightarrow F$. If $F = \RR$, we will write $\Distrib(\RR^n) = \Distrib(\RR^n, \RR)$. We will simply write $u(\varphi)$ for the application of a distribution $u \in \Distrib(\RR^n,F)$ to $\varphi \in C_c^\infty(\RR^n)$.

Moreover, let $\supp u \subseteq \RR^n$ denote the support of $u \in \Distrib(\RR^n)$ and $\DistribC(\RR^n, F)$ the space of all distributions with compact support. It is well-known that every distribution with compact support can be extended to $C^\infty(\RR^n)$.

\subsection{General Construction}
Our construction is very much inspired by results of Goodey and Weil \cite{Goodey1984}, who defined the Goodey--Weil distributions of real-valued valuations on convex bodies, and relies on the following analogous result for continuous, dually epi-translation invariant valuations on convex functions proven in \cite{Knoerr2020a}. Note that for continuous, additive maps defined on $\Conv(\RR^n,\RR)$ (which are one-homogeneous valuations), the additional invariance property that is used in \cite{Knoerr2020a} may be dropped without affecting the results.

\begin{theorem}[\cite{Knoerr2020a}*{Thm.~2}]\label{thm:exGWCompSupp}
 Let $F$ be a locally convex vector space and $\overline{F}$ its completion. For every continuous and additive $\mu: \Conv(\RR^n,\RR) \rightarrow F$ there exists a uniquely determined distribution $\GW(\mu) \in \Distrib(\RR^n,\overline{F})$ which satisfies
 \begin{align}\label{eq:GWonCinfbyPsi}
 \GW(\mu)(\phi) = \mu(f+\phi) - \mu(f),
 \end{align}
 for every $\phi \in C_c^\infty(\RR^n)$ and $f \in \Conv(\RR^n,\RR)$ such that $f + \phi \in \Conv(\RR^n,\RR)$, and $\GW(\mu)$ determines $\mu$ uniquely. %
 If $F$ admits a continuous norm, then $\GW(\mu)$ has compact support and $\GW(\mu)[f]=\mu(f)$ for all $f\in\Conv(\RR^n,\RR)\cap C^\infty(\RR^n)$.%
\end{theorem}

\noindent
We will apply Theorem~\ref{thm:exGWCompSupp} in two settings: for $F=\RR$ and for $F= C(W)$, the space of all continuous functions on a locally compact space $W$ equipped with the topology of uniform convergence on compact subsets. Suppose that $\Psi:\Conv(\RR^n,\RR) \rightarrow C(W)$ is additive and continuous, and consider for $x \in W$ the map
\begin{align*}
 \Psi_x: \begin{cases}
          \Conv(\RR^n,\RR) \rightarrow \RR,\\
          \Psi_x(f) = \Psi(f)[x].
         \end{cases}
\end{align*}
Then $\Psi_x$ clearly is a continuous and additive functional and therefore satisfies the conditions of Theorem~\ref{thm:exGWCompSupp}. By taking $W=\RR^n$ and the (continuous) inclusion $\Conv(\RR^n, \RR) \subseteq C(\RR^n)$ (compare Lemma \ref{lem:epiConvEquiv}\eqref{lem:epiConvEquivUnifComp}), we obtain a family of distributions with compact support:
\begin{definition}
 Let $\Psi: \Conv(\RR^n,\RR) \rightarrow \Conv(\RR^n,\RR)$ be continuous and additive. Then $(\GW(\Psi_x))_{x \in \RR^n}$ is the \emph{associated family of Goodey--Weil distributions} of~$\Psi$.
\end{definition}

\noindent
Coming from a map $\Psi:\Conv(\RR^n,\RR) \to \Conv(\RR^n, \RR)$, the functionals $\Psi_x$ depend continuously on $x \in \RR^n$. To make this more precise, consider for compact $A \subseteq \RR^n$
\begin{align*}
 \Psi_A: \begin{cases}
          \Conv(\RR^n,\RR) \rightarrow C(A),\\
          \Psi_A(f)[x] = \Psi(f)[x].
         \end{cases}
\end{align*}
The map $\Psi_A$ again satisfies the conditions of Theorem~\ref{thm:exGWCompSupp}, where $C(A)$ is equipped with the maximum norm. Note that the restriction to a compact subset $A$ is needed here to obtain a Goodey--Weil distribution $\GW(\Psi_A)$ with compact support.

Denoting by $i_x:C(A) \rightarrow \RR$, $x \in A$, the continuous evaluation map, we see that $\GW(\Psi_x) = i_x \circ \GW(\Psi_A)$. As a consequence, we get the following bound on $\supp \GW(\Psi_x) \subseteq \supp\GW(\Psi_A)$, $x \in A$.
\begin{proposition}
 \label{propLocalBoundSupport}
 Suppose that $\Psi:\Conv(\RR^n,\RR)\rightarrow C(\RR^n)$ is continuous and additive, where $C(\RR^n)$ is equipped with the topology of uniform convergence on compact subsets.
 Then for every compact $A\subseteq \RR^n$ there exists a compact $A'\subseteq \RR^n$ such that
 \begin{align*}
    \supp \GW(\Psi_x)\subset A',\quad\forall x\in A.
 \end{align*}
\end{proposition}

In the one-dimensional case, the (continuous) relation between the distributions is made explicit in Theorem~\ref{mthm:Char1Dim}.

\medskip

Another way to connect the associated Goodey--Weil distributions $\GW(\Psi_x)$ at different points $x\in\RR^n$ are equivariance properties of $\Psi$. Recall that the standard representation of $\GL(n)$ on functions extends to $\Distrib(\RR^n)$ by $(\eta \cdot u)(\varphi) = u(\eta^{-1} \cdot \varphi)$, $u \in \Distrib(\RR^n), \varphi \in C_c^\infty(\RR^n)$ and $\eta \in \GL(n)$. %

\begin{lemma}\label{lem:equivPropGWdistr}
 Suppose that $\Psi: \Conv(\RR^n,\RR) \to \Conv(\RR^n,\RR)$ is continuous and additive and let $G\subseteq \GL(n)$ be a subgroup.
 
 If $\Psi$ is $G$-equivariant, then for every $x \in \RR^n$ and $\eta \in G$,
 \begin{align*}
  \GW(\Psi_{\eta(x)})[\phi] = \left(\eta \cdot \GW(\Psi_x)\right)[\phi] = \GW(\Psi_x)[\eta^{-1}\cdot\phi], \quad \phi \in C^\infty_c(\RR^n).
 \end{align*}
 In particular, $\GW(\Psi_x)$ is invariant under $G_x$, the stabilizer of $x$ in $G$.
\end{lemma}

\begin{proof}
 The proof is an easy consequence of \eqref{eq:GWonCinfbyPsi}.
\end{proof}

\noindent
Lemma~\ref{lem:equivPropGWdistr} implies in particular that there is essentially one Goodey--Weil distribution for every orbit of $G$ for $G$-equivariant endomorphisms. If $G=\GL(n)$ or $G$ is the subgroup generated by $\SO(n)$ and dilations, there are exactly two orbits, namely $\{0\}$ and $\RR^n\setminus \{0\}$, where the values on the $\{0\}$-orbit are already determined by continuity. This reduces the problem to analyzing only one distribution.

\subsection{Auxiliary Results}
In this section we prove some auxiliary results for endomorphisms, the first one showing that the associated Goodey--Weil distributions of monotone endomorphisms are actually non-negative measures.

\begin{lemma}\label{lem:GWdistrMonot}
 Suppose that $\Psi:\Conv(\RR^n,\RR) \rightarrow \Conv(\RR^n,\RR)$ is continuous and additive. Then $\Psi$ is monotone if and only if the Goodey--Weil distributions $\GW(\Psi_x)$, $x \in \RR^n$, are given by non-negative measures $\mu_x$ with compact support. In this case, 
 \begin{align}\label{eq:LemGWDistrMonDarst}
  \Psi(f)[x] = \int_{\RR^n} f d \mu_x,\quad x \in \RR^n,
 \end{align}
 for every $f \in \Conv(\RR^n,\RR)$.
\end{lemma}
\begin{proof}
 Assume first that $\Psi$ is monotone and let $\phi_1\le \phi_2 \in C_c^\infty(\RR^n)$. Take $f\in \Conv(\RR^n,\RR)$ such that $f+\phi_i$ is convex for $i=1,2$ (a possible choice would be $f(y)=c\|y\|^2$ for $c>0$ large enough; see, e.g., \cite{Knoerr2020a}*{Lem. 5.1}). Then $f+\phi_1\le f+\phi_2$ and, by the monotonicity of $\Psi$ and \eqref{eq:GWonCinfbyPsi},
\begin{align*}
    \GW(\Psi_x)[\phi_1]=\Psi_x(f+\phi_1)-\Psi_x(f)\le \Psi_x(f+\psi_2)-\Psi_x(f)=\GW(\Psi_x)[\phi_2].
\end{align*}
As $\phi_1$ and $\phi_2$ were chosen arbitrarily, $\GW(\Psi_x)$ is a positive distribution and, hence, given by a non-negative measure $\mu_x$. By Theorem~\ref{thm:exGWCompSupp}, $\supp\GW(\Psi_x)$ is compact, so $\mu_x$ is compactly supported as well, and 
\begin{align*}
    \Psi_x(f)=\GW(\Psi_x)[f]=\int_{\RR^n} f d\mu_x,\quad f\in\Conv(\RR^n,\RR)\cap C^\infty(\RR^n).
\end{align*}
Obviously the right-hand side of this equation extends by continuity to $\Conv(\RR^n,\RR)$, so this representation holds for all $f\in\Conv(\RR^n,\RR)$.

If, on the other hand, the Goodey--Weil distributions are given by non-negative measures $\mu_x$, then the previous argument can be repeated to obtain \eqref{eq:LemGWDistrMonDarst}, from which it is clear that $\Psi$ is monotone.
\end{proof}

\noindent
Despite their very simple proofs, the following lemmas turn out to be very useful in proving characterization results. We start with the following observation.

\begin{lemma}\label{lem:AffFctsMapToAff}
	Suppose that $\Psi: \Conv(\RR^n,\RR) \rightarrow \Conv(\RR^n,\RR)$ is additive. Then $\Psi$ maps affine functions to affine functions.
\end{lemma}
\begin{proof}
	As $\Psi$ is additive, the constant zero function is mapped to itself. Now let $g:\RR^n  \to \RR$ be affine. Then $-g$ is affine as well and
	\begin{align*}
	\Psi (g) + \Psi(-g) = \Psi(0) = 0
	\end{align*}
	implies that $\Psi(-g) = -\Psi(g)$. Thus $\Psi(g)$ and $-\Psi(g)$ are convex, so $\Psi(g)$ must be affine.
\end{proof}

For a monotone endomorphism $\Psi$, Lemma~\ref{lem:AffFctsMapToAff} has the following implication for the total mass of the measures $\mu_x$ defining the Goodey--Weil distributions $\GW(\Psi_x)$. Let $1$ denote the constant function $x\mapsto 1$.
\begin{corollary}
	If $\Psi:\Conv(\RR^n,\RR)\rightarrow\Conv(\RR^n,\RR)$ is continuous, additive and monotone with associated measures $(\mu_x)_{x\in \RR^n }$, then $\mu_x(\RR^n)=\Psi(1)[0]$, for all $x\in \RR^n $. In particular,  $\Psi=0$ if and only if $\Psi(1)[0]=0$.
\end{corollary}
\begin{proof}
	By Lemma \ref{lem:AffFctsMapToAff}, $\Psi(1)$ is an affine function. Because $\Psi(1)\ge \Psi(0)=0$, $\Psi(1)$ has to be constant and, hence,
	\begin{align*}
		\mu_x(\RR^n)=\int_{\RR^n}1d\mu_x=\Psi(1)[x]=\Psi(1)[0],\quad\text{for all } x\in \RR^n.
	\end{align*}%
	By the uniqueness of $\GW(\Psi_x)$ (by Theorem~\ref{thm:exGWCompSupp}), $\Psi$ is the zero map if and only if $\GW(\Psi_x)=\mu_x=0$ for all $x\in \RR^n $. As $\mu_x$ is a non-negative measure, this is the case if and only if $\mu_x(\RR^n)=0$ for all $x\in \RR^n $, that is, if and only if $\Psi(1)[0]=0$.
\end{proof}
\noindent
We now turn to endomorphisms $\Psi:\ConvO(\RR^n)\rightarrow\ConvO(\RR^n)$.
\begin{lemma}\label{lem:restrictRadEquiv}
	Suppose that $\Psi:\ConvO(\RR^n) \rightarrow \ConvO(\RR^n)$ is continuous and  additive. If $\Psi(0)$ is finite everywhere, then  $\Psi$ restricts to an endomorphism of $\Conv(\RR^n, \RR)$. This holds in particular if $\Psi$ is radially equivariant.
\end{lemma}
\begin{proof}
	Because $\Psi$ is additive, $\Psi(0)=\Psi(0)+\Psi(0)$, which therefore can only take the values $0$ or $+\infty$. As $\Psi(0)$ is convex, it must be the convex indicator function of a convex set $A\subset \RR^n$.
	
	Next, assume that $\Psi(0)$ is finite (and, hence, $0$) and let $f\in\Conv(\RR^n,\RR)$ be given. Then $\frac{1}{j}f$ epi-converges to $0$ for $j\rightarrow\infty$ by Lemma~\ref{lem:epiConvEquiv}\eqref{lem:epiConvEquivPWonDense}. The continuity of $\Psi$ thus implies that $\Psi(\frac{1}{j}f)=\frac{1}{j}\Psi(f)$ epi-converges $\Psi(0)=0$. Again, Lemma \ref{lem:epiConvEquiv}\eqref{lem:epiConvEquivPWonDense} implies that this convergence is pointwise on the interior of the domain of the zero function, so $\frac{1}{j}\Psi(f)[x]$ converges to zero for all $x\in \RR^n$. In particular, $\Psi(f)[x]$ must be finite for all $x\in \RR^n $, that is, $\Psi(f)\in\Conv(\RR^n,\RR)$.

	Finally, assume that $\Psi$ is radially equivariant. As $\Psi(0)\in\ConvO(\RR^n)$, $0$ is contained in the interior of $A = \dom \Psi(0)$. Given $x\in \RR^n $, we can thus choose $r>0$ such that $rx\in A$. Because $\Psi$ is radially equivariant, this implies
	\begin{align*}
	0=\mathbbm{1}^\infty_A(rx)=\Psi(0)[rx]=\Psi(0)[x].
	\end{align*}
	As this holds for all $x\in \RR^n $, $\Psi(0)=0$.
\end{proof}

\subsection{Examples}
\label{subsection:examples}
In this section, we give additional examples of endomorphisms on convex functions that do not appear in the characterization results of Theorems~\ref{mthm:CharGLEquiv} to \ref{mthm:Char1Dim}. In particular, we want to emphasize that in the two considered equivariant cases all endomorphisms on $\Conv(\RR^n,\RR)$ extend by continuity to $\ConvO(\RR^n)$. Moreover, the equivariance implies that endomorphisms of $\ConvO(\RR^n)$ restrict to endomorphisms of $\Conv(\RR^n,\RR)$, that is, every such endomorphism maps finite functions to finite functions. As we will see, both statements do not hold in general.

\medskip

In order to give the first example, we make use of \emph{Monge--Amp\`ere measures}, coming from solutions of the Monge--Amp\`ere equation (see, e.g., \cites{Figalli2017, Gutierrez2016, Trudinger2008, Aleksandrov1958}). Monge--Amp\`ere measures have been used recently in the theory of valuations on convex functions \cites{Alesker2019, Colesanti2017b,Colesanti2019b, Colesanti2021b}, where the related notion of \emph{Hessian measure} was used (see, e.g., \cite{Colesanti2005} and the references therein). As we will not need the exact definition, we will just state the properties we use and refer to \cites{Figalli2017, Gutierrez2016, Colesanti2019b, Colesanti2021b} for more details. 

The following theorem is due to Aleksandrov \cite{Aleksandrov1958} (see also \cite{Figalli2017}*{Thm.~2.3, Prop.~2.6 and Thm.~A.31}), where the last two points are direct consequences of the first two. Recall that a Radon measure on $\RR^n$ is a Borel measure that is finite on compact sets, and let $D^2 f$ denote the Hessian matrix of $f \in C^2(\RR^n)$.
\begin{theorem}\label{thm:defMongAmp}
 For every $f \in \Conv(\RR^n, \RR)$ there exists a Radon measure $\MongAmp{f}{\cdot}$ on $\RR^n$, the \emph{Monge--Amp\`ere measure} of $f$, such that the following holds:
 \begin{enumerate}
  \item If $f \in \Conv(\RR^n, \RR) \cap C^2(\RR^n)$, then $\MongAmp{f}{\cdot}$ is absolutely continuous with respect to the Lebesgue measure on $\RR^n$ and 
  \begin{align*}
   d\MongAmp{f}{x} = \det (D^2 f(x)) dx, \quad x \in \RR^n.
  \end{align*}
  \item If $f_j \in \Conv(\RR^n, \RR)$ epi-converges to $f \in \Conv(\RR^n, \RR)$, then $\MongAmp{f_j}{\cdot}$ converges weakly to $\MongAmp{f}{\cdot}$.
  \item $\MongAmp{f}{\cdot}$ is a non-negative measure.
  \item If $n=1$ and $g \in \Conv(\RR^n,\RR)$, then $\MongAmp{f+g}{\cdot} = \MongAmp{f}{\cdot} + \MongAmp{g}{\cdot}$.
 \end{enumerate}

\end{theorem}

\noindent
The Monge--Amp\`ere measure will also be used in Section~\ref{sec:proof1Dim} in the proof of Theorem~\ref{mthm:Char1Dim}.

\medskip

\noindent
We now have all prerequisites to give the first example of an endomorphism on $\Conv(\RR, \RR)$ that does not extend to $\ConvO(\RR)$.

\begin{example}\label{ex:1dMongAmpKonst}
 Let $g \in \Conv(\RR,\RR)$ and $\zeta \in C_c(\RR)$ non-negative. Then, by Theorem~\ref{thm:defMongAmp}, the map $\Psi$ defined by
 \begin{align*}
  \Psi(f)[x] = g(x) \cdot \int_{\RR} \zeta(|y|) d\MongAmp{f}{y}, \quad x \in \RR,%
 \end{align*}
 is continuous, $\Psi(f) \in \Conv(\RR,\RR)$ as $\MongAmp{f}{\cdot}$ is non-negative and finite on $\supp \zeta$ for every $f \in \Conv(\RR,\RR)$, and $\Psi$ is additive, since $\MongAmp{f}{\cdot}$ is additive. $\Psi$ does not possess a continuous extension to $\ConvO(\RR)$. This can be seen by approximating convex indicator functions, e.g., with polynomials of degree four or higher, for which the second derivatives diverge. Moreover, the example can be extended to arbitrary dimensions by $\hat \Psi(f)[\cdot] = \Psi(s \mapsto f(sy))[\mathrm{pr}_{\langle y\rangle}(\cdot)]$ for some fixed $y \in \RR^n$, where $\mathrm{pr}_{\langle y\rangle}$ denotes the orthogonal projection onto the linear span $\langle y\rangle$ of $y$.
\end{example}

\medskip

Recall that Lemma \ref{lem:restrictRadEquiv} provides a very simple description of the endomorphisms of $\ConvO(\RR^n)$ that map finite-valued functions to finite-valued functions. Of course, it is very easy to construct maps that do not have this property: We can always just add the convex indicator function of a ball around the origin to any endomorphism. By Lemma~\ref{lem:epiConvEquiv}, the new map is still continuous as a map from $\Conv(\RR^n,\RR)$ to $\ConvO(\RR^n)$, but obviously its image does not contain any finite-valued function. Let us construct an example that is slightly less artificial:

\begin{example}\label{ex:IntPhiT}
 Let $\varphi \in \Conv(\RR,\RR)\cap C^2(\RR)$ be even and non-negative and define
 \begin{align}\label{eq:exIntPhiT}
  \Psi_\varphi(f)[t] = \int_{-\varphi(t)}^{\varphi(t)} (f(s)-f(0)) ds, \quad f \in \Conv(\RR,\RR).
 \end{align}
 Then the second derivative of $\Psi_\varphi(f)$, $f \in \Conv(\RR,\RR)\cap C^2(\RR)$, is given for $t \in \RR$ by
 \begin{align*}
  \Psi_\varphi(f)''[t] = \varphi''(t)\left( f(\varphi(t)) + f(-\varphi(t)) - 2f(0)\right)
               +(\varphi'(t))^2(f'(\varphi(t)) - f'(-\varphi(t))).
 \end{align*}
 Here, the first term is non-negative by the convexity of $\varphi$ and $f$, and the second term is non-negative since $f'$ is increasing, again by convexity. Hence, $\Psi_\varphi(f)$ is convex for every $f \in \Conv(\RR,\RR)\cap C^2(\RR)$ and thus for every $f \in \Conv(\RR,\RR)$ by approximation. Moreover, note that $\Psi_\varphi$ is obviously continuous and additive, so $\Psi_\varphi$ is an endomorphism of $\Conv(\RR, \RR)$, and that, by approximation, the condition $\varphi \in C^2(\RR)$ is actually not necessary in \eqref{eq:exIntPhiT}.
 
 \medskip
 
 The map $\Psi_\varphi$ can be extended to $\ConvO(\RR)$ if and only if $\varphi(0) = 0$: Indeed, \eqref{eq:exIntPhiT} can be used as definition for $f \in \ConvO(\RR)$, since every such $f$ is bounded from below on the compact subsets $[-\varphi(t),\varphi(t)]$, $t \in \RR$, by semi-continuity. Moreover, we can choose a sequence $f_j \in \Conv(\RR,\RR)$ that epi-converges monotonously to $f$, and, by monotone convergence and Lemma~\ref{lem:epiConvEquiv} (using that $0 \in \interior \dom f$),
 \begin{align*}
  \Psi_\varphi(f_j)[t] = \int_{-\varphi(t)}^{\varphi(t)} f_j(s) ds - 2 f_j(0) \varphi(t) \rightarrow \int_{-\varphi(t)}^{\varphi(t)} f(s)ds -2 f(0) \varphi(t) = \Psi_\varphi(f)[t],
 \end{align*}
 for every $t \in \RR$, as $j\to \infty$. The resulting function $\Psi_\varphi(f)$ is therefore convex as pointwise limit of convex function and it is finite in a neighborhood of $0$ if and only if $\varphi(t) \to 0$ as $t\to 0$, which is equivalent to $\varphi(0) = 0$ by continuity. The proof that $\Psi_\varphi$ is continuous on $\ConvO(\RR)$, finally, is very similar to the according part of the proof of Theorem~\ref{mthm:CharGLEquiv} and will therefore be omitted at this point.

 \medskip
 
 Another way to extend this example is to consider more generally $\varphi \in \ConvO(\RR)$ even and non-negative. The naive idea to just extend \eqref{eq:exIntPhiT}, however, leads to a map that is not continuous anymore. Indeed, affine functions would be mapped to the constant zero function, while every other function would be infinite on the complement of $\overline{\dom \varphi}$. For this reason, we use the following definition for $\varphi \in \ConvO(\RR)$, $\varphi$ even and non-negative, and $f \in \Conv(\RR,\RR)$:
 \begin{align*}
  \Psi_\varphi(f)[t] = \begin{cases}
                        \int_{-\varphi(t)}^{\varphi(t)} (f(s)-f(0)) ds, & t \in \interior \dom \varphi,\\
                        \liminf_{t' \to t, t' \in \interior \dom \varphi} \Psi_\varphi(f)[t'], & t \in \partial \dom \varphi,\\
                        \infty, & t \in \RR \setminus \overline{\dom \varphi}.
                       \end{cases}
 \end{align*}
 Note that $\Psi_\varphi(f)$ is by definition lower semi-continuous and finite in the neighborhood $\interior \dom \varphi$ of the origin. To show that $\Psi_\varphi(f)$ is convex, first let $t \in \interior \dom \varphi$. Since $\varphi$ is convex and finite in a neighborhood of $[-t,t]$, we can find an affine function $g$ such that its graph is a supporting hyperplane of $\epi \varphi$ at $t$. We may therefore define $\widetilde\varphi\in \Conv(\RR^n,\RR)$ by requiring $\widetilde{\varphi}$ to be equal to $\varphi$ on $[-t,t]$, equal to $g(s)$ for $s>t$ and equal to $g(-s)$ for $s<-t$. In particular, the first part of the example implies that $\Psi_{\widetilde \varphi}(f)$ is convex and as $\Psi_\varphi(f)[t'] = \Psi_{\widetilde\varphi}(f)[t']$ for $|t'|<|t|$, $\Psi_\varphi(f)$ is convex on $\interior \dom \varphi$. Noting, finally, that since any linear function $h$ is odd,
  \begin{align*}%
  \int_{-\varphi(t)}^{\varphi(t)} f(s) - f(0) ds = \int_{-\varphi(t)}^{\varphi(t)} f(s) - f(0) - h(s) ds, \quad t \in \interior \dom \varphi,
 \end{align*}
 so we can assume that the integrand is non-negative (take, e.g., the graph of $h$ to be a supporting hyperplane of the epigraph of $f(s)-f(0)$ at $0$). Consequently, monotone convergence implies that $\Psi_\varphi(f)[t] = \lim_{t' \to t} \Psi_\varphi(f)[t']$, where we take $t' \in \interior \dom \varphi$ and $t \in \partial\dom\varphi$, and therefore $\Psi_\varphi(f)$ is convex on all of $\RR$, that is, $\Psi_\varphi(f) \in \ConvO(\RR)$.
 
 The continuity of the map $\Psi_\varphi: \Conv(\RR,\RR)\to\ConvO(\RR)$ follows from Lemma~\ref{lem:epiConvEquiv} using that every $\Psi_\varphi(f)$ is constant on the complement of $\overline{\dom\varphi}$ and that, for $t \in \interior \dom \varphi$, $\Psi_\varphi(f)[t]$ depends only on the values of $f$ on the compact set $[-\varphi(t),\varphi(t)]$.

\end{example}

\begin{example}
 The previous Example~\ref{ex:IntPhiT} can be generalized to arbitrary dimensions by defining for $y \in \RR^n  \setminus\{0\}$ and $\varphi \in \ConvO(\RR)$
 \begin{align*}
  \widetilde\Psi_{\varphi,y}(f)[x] = \Psi_\varphi(s \mapsto f(sy))[\mathrm{pr}_{\langle y\rangle}(x)], \quad f \in \Conv(\RR^n,\RR), x \in \RR^n ,
 \end{align*}
 where $\mathrm{pr}_{\langle y\rangle}$ denotes the orthogonal projection onto $\langle y\rangle$. Moreover, taking means over $y \in \RR^n  \setminus\{0\}$ and a family of $\varphi\in\ConvO(\RR)$ yields further examples that are not ``1-dimensional'' anymore.
\end{example}


\medskip

Examples~\ref{ex:1dMongAmpKonst} and \ref{ex:IntPhiT}, as well as the results of Theorems~\ref{mthm:CharGLEquiv} and \ref{mthm:CharMonRadEquiv} lead to the following (in our opinion very interesting) question: Is there a general criterion when a continuous and additive map $\Psi:\Conv(\RR^n,\RR)\rightarrow\ConvO(\RR)$ can be extended to $\ConvO(\RR^n)$? For dually epi-translation invariant valuations on convex functions, which are closely related to additive maps, some results in this direction were obtained in \cite{Knoerr2020a} in terms of the supports of Goodey-Weil distributions. A similar approach can be used to obtain conditions under which an endomorphism can be extended ``pointwise", but the last part of Example \ref{ex:IntPhiT} suggests that there may be additional obstructions imposed by continuity.

\medskip

In some sense, an answer to this question would yield a structural relation between the space $\ConvO(\RR^n)$ and its dense subspace $\Conv(\RR^n, \RR)$, which would be interesting on its own. Moreover, it could be helpful in the characterization of endomorphisms with different equivariance properties, as our proofs rely on the fact that we can restrict to $\Conv(\RR^n,\RR)$ and work with finite convex functions.%

\bigskip

\section{Proof of Theorem~\ref{mthm:CharGLEquiv} and Theorem~\ref{mcor:CharGLEquivWholeConv}}

\noindent
Next, we prove Theorems~\ref{mthm:CharGLEquiv} and \ref{mcor:CharGLEquivWholeConv}, starting with the ``if''-part of Theorem~\ref{mthm:CharGLEquiv}.

\begin{theorem}\label{thm:GLequivCandidate}
 Suppose that $\nu\in\MeasC^+(\RR)$ with $\int_{\RR^\times}|s|^{-1} d\nu(s) < \infty$ and $c \in \RR$.
 
 Then there exists a unique continuous, additive and $\GL(n)$-equivariant map $\Psi: \ConvO(\RR^n) \rightarrow \ConvO(\RR^n)$ satisfying
 \begin{align}\label{eq:DefCandidateGLEquiv}
  \Psi(f)[x] = cf(0) + \int_{\RR^\times}\!\!\! \frac{f(sx)-f(0)}{|s|^2}\, d\nu(s), \quad x \in \RR^n ,
 \end{align}
 for every $f \in \Conv(\RR^n,\RR)$.
\end{theorem}
\begin{proof}
 As $\Conv(\RR^n,\RR)$ is a dense subspace of $\ConvO(\RR^n)$, $\Psi$ is uniquely determined by \eqref{eq:DefCandidateGLEquiv} and continuity, if it exists. To construct the desired map, observe first that we may without loss of generality assume that $c=0$. Indeed, the map $f \mapsto cf(0)$ has all the claimed properties, which are preserved when adding two such maps.
 
 \bigskip
 
 Let $\nu\in\MeasC^+(\RR)$ with $\int_{\RR^\times}|s|^{-1} d\nu(s) < \infty$ be given and set $a:=\inf\supp\nu$, $b:=\sup\supp\nu$. For $f\in\ConvO(\RR^n)$ define $D_f:=\{x\in \RR^n : [a,b] \cdot x \subseteq \mathrm{int}\dom f\}$, which is an open and convex neighborhood of $0$. Let us set
 \begin{align}\label{eq:prfGLequivCandidErweiterung}
 	\Psi(f)[x]:=\begin{cases}
 	    \int_{\RR^\times}\frac{f(sx)-f(0)}{|s|^2}d\nu(s), & x\in D_f,\\
 		\liminf\limits_{x'\rightarrow x,  x'\in D_f} \int_{\RR^\times}\frac{f(sx')-f(0)}{|s|^2}d\nu(s), &x\in\partial D_f,\\
 		\infty, & x\in \RR^n \setminus\overline{D_f}.
 	\end{cases}
 \end{align}
 Note that, if well-defined, $\Psi(f)$ is lower semi-continuous by construction and this definition coincides with \eqref{eq:DefCandidateGLEquiv} on $\Conv(\RR^n,\RR)$. In order to show well-definedness, let $x\in D_f$ and observe that as $[a,b]\cdot x$ is compact and contained in the open and convex set $\mathrm{int}\dom(f)$, $\mathrm{int}\dom (f)$ contains a compact and convex neighborhood of $[a,b]\cdot x$. By Proposition \ref{prop:uniformLipschitz}, $f$ is Lipschitz continuous on $[a,b]\cdot x$ with constant $C>0$, and we can estimate
 \begin{align*}
 \frac{|f(sx)-f(0)|}{|s|^2} \leq  C\frac{\|x\|}{|s|}, \quad s\in [a,b]\setminus \{0\},
 \end{align*}
 where the right-hand side is $\nu$-integrable on $\RR^\times$ by assumption. Consequently, the integral
 \begin{align}\label{eq:prfThmCandidateGLInteg}
  \int_{\RR^\times}\frac{f(sx)-f(0)}{|s|^2}d\nu(s)
 \end{align}
 converges, that is, $\Psi(f)[x]$ is well-defined and finite for $x\in D_f$. Moreover, the convexity of $f$ implies the existence of $\alpha>0, z \in \RR^n $ such that $f(y) \geq \alpha \langle y,z\rangle + f(0)$ for every $y \in \RR^n $. Hence, the integrand of \eqref{eq:prfThmCandidateGLInteg} is bounded from below uniformly in $x \in \RR^n $ by $-\alpha|\langle x,z\rangle|\cdot |s|^{-1}$, which implies that $\Psi(f)[x]$ exists for $x \in \partial D_f$ and is not equal to $-\infty$. Overall, we have shown that $\Psi(f):\RR^n  \rightarrow (-\infty, \infty]$ is a well-defined and lower semi-continuous function, which is finite on the neighborhood $D_f$ of $0$.
 
 \bigskip
 
 Let us now show that $\Psi(f)\in\ConvO(\RR^n)$, that is, that $\Psi(f)$ is convex. Note that $\Psi(f)$ is given by \eqref{eq:prfThmCandidateGLInteg} on the open and convex set $D_f$, so it is in particular convex and, hence, continuous on this set. As $\Psi(f)$ is infinite outside $\overline{D_f}$, it is sufficient to prove
 \begin{align}\label{eq:prfThmCandidateGLConvex}
 	\Psi(f)[\lambda x+(1-\lambda)y]\le \lambda\Psi(f)[x]+(1-\lambda)\Psi(f)[y]\quad\text{for }\lambda\in(0,1)
 \end{align}
 for $x\in D_f$, $y\in\partial D_f$ and for $x,y \in \partial D_f$. For the first case, let $y_j \to y$ be a sequence in $D_f$ such that $\Psi(f)[y_j] \to \Psi(f)[y]$, $j\to \infty$. Then, by the convexity of $D_f$, $\lambda x+(1-\lambda)y_j\in D_f$ for all $j\in\NN$, and $\lambda x+(1-\lambda)y\in D_f$ for $\lambda\in (0,1)$, as $D_f$ is also open. The continuity and convexity of $\Psi(f)$ on $D_f$ therefore imply \eqref{eq:prfThmCandidateGLConvex},
 \begin{align}\label{eq:prfThmCandidateGLConvexBnd}
 	\Psi(f)[\lambda x+(1-\lambda)y]=&\lim\limits_{j\rightarrow\infty}\Psi(f)[\lambda x+(1-\lambda)y_j] \nonumber\\
 	\le& \lim\limits_{j\rightarrow\infty}\! \lambda\Psi(f)[x]\!+\!(1-\lambda)\Psi(f)[y_j]
 	= \lambda\Psi(f)[x]+(1-\lambda)\Psi(f)[y].
 \end{align}
 The second case, $x,y\in\partial D_f$, follows from the first case by taking a sequence $x_j \in D_f$ with $x_j \to x$ such that $\Psi(f)[x_j] \to \Psi(f)[x]$. Then, by lower semi-continuity and \eqref{eq:prfThmCandidateGLConvexBnd},
 \begin{align*}
  &\Psi(f)[\lambda x +(1-\lambda)y] \leq \liminf_{j \to \infty} \Psi(f)[\lambda x_j + (1-\lambda)y]\\
                                   \leq& \liminf_{j \to \infty} \lambda \Psi(f)[x_j] + (1-\lambda)\Psi(f)[y] = \lambda\Psi(f)[x]+(1-\lambda)\Psi(f)[y],
 \end{align*}
 that is, $\Psi(f)\in\ConvO(\RR^n)$.
 
 \bigskip
 
 Next, we prove that $\Psi$ is continuous with respect to the topology induced by epi-convergence. Let $f_j \in \ConvO(\RR^n)$ be an epi-convergent sequence with limit $f\in\ConvO(\RR^n)$. In order to show $\Psi(f_j) \rightarrow \Psi(f)$, by Lemma~\ref{lem:epiConvEquiv}(\ref{lem:epiConvEquivPWonDense}), it is sufficient to show that $\Psi(f_j)[x_0]$ converges to $\Psi(f)[x_0]$ for all $x_0\in \RR^n \setminus\partial D_f$.

 First suppose that $x_0 \in D_f$. As $D_f$ is open, there exists $\varepsilon>0$ such that $B_{\varepsilon}(x_0)\subset D_f$, so the compact set $A:=\{\lambda x: \lambda\in[a,b],x\in B_\varepsilon(x_0)\}$ is contained in $\interior\dom f$. Consequently, by Lemma~\ref{lem:epiConvEquiv}(\ref{lem:epiConvEquivUnifComp}), the sequence $(f_j)_j$ converges uniformly to $f$ on $A$. In particular, there exists $j_0\in\NN$ such that $f_j$ is bounded on $A$ by a constant not depending on $j\ge j_0$. Hence, Proposition~\ref{prop:uniformLipschitz} implies that the Lipschitz constants of the functions $f_j$ are bounded on $[a,b]\cdot x_0$ by some constant $C>0$ not depending on $j\ge j_0$, as $A$ is a compact neighborhood of $[a,b]\cdot x_0$. Thus for $j\ge j_0$
 \begin{align}\label{eq:prfThmAContDfEstLip}
 	\frac{|f_j(sx_0)-f_j(0)|}{|s|^2}\le C\frac{\|x_0\|}{|s|},\quad s\in[a,b]\setminus \{0\},
 \end{align}
 that is, the integrands are uniformly bounded on $\supp \nu \setminus \{0\}$ by a $\nu$-integrable function. Dominated convergence and uniform convergence of $f_j$ on $A \supseteq \supp \nu \cdot x_0$ then imply
 \begin{align*}
 	\lim\limits_{j\rightarrow\infty}\Psi(f_j)[x_0]=\lim\limits_{j\rightarrow\infty}\int_{\RR^\times}\!\!\!\frac{f_j(sx_0)-f_j(0)}{|s|^2}d\nu(s)=\int_{\RR^\times}\!\!\!\frac{f(sx_0)-f(0)}{|s|^2}d\nu(s)=\Psi(f)[x_0].
 \end{align*}
 
 \medskip
 
 Suppose now that $x_0\in \RR^n \setminus\overline{D_f}$. Then $[a,b]\cdot x_0\setminus \overline{\dom f}\ne\emptyset$, so there exist $\varepsilon, \eta >0$ such that either $[b-\varepsilon,b]\cdot \lambda x_0\subseteq \RR^n \setminus \overline{\dom f}$ or $[a,a+\varepsilon]\cdot \lambda x_0\subseteq \RR^n \setminus \overline{\dom f}$ or both for all $\lambda \in [1-\eta,1]$. Pick one of the subsets $[a,a+\varepsilon]$, $[b-\varepsilon,b]$ with this property and denote it by $B$. Then $\nu(B)>0$ by the definition of $a$ and $b$. We can further choose $\delta>0$ such that $[-2\delta,2\delta]\cdot x_0 \subseteq \interior \dom f$, as $0 \in \interior\dom f$, and then $[-\delta,\delta] \cap B = \emptyset$.
 
 As we need to show that $\Psi(f_j)[x_0] \to \infty$, we will use the (disjoint) decomposition $\RR = [-\delta,\delta] \cup B \cup \RR\setminus([-\delta,\delta]\cup B)$ to give a diverging lower bound of $\Psi(f_j)[x_0]$.
 
 First, since $[-2\delta,2\delta]\cdot x_0 \subseteq \interior \dom f$, the sequence $(s \mapsto f_j(sx_0))_{j\in \NN}$ converges uniformly to $s \mapsto f(sx_0)$ on $[-2\delta,2\delta]$. Hence, Proposition~\ref{prop:uniformLipschitz} implies that these functions are Lipschitz continuous on $[-\delta,\delta]$ with Lipschitz constant bounded by some $C>0$ independent of $j \in \NN$. In particular, we obtain an estimate 
 \begin{align}\label{eq:prfThmACandEstDelta}
  \frac{|f_j(sx_0)-f_j(0)|}{|s|^2} \leq  C\frac{\|x_0\|}{|s|}, \quad s\in [-\delta,\delta] \setminus \{0\} \text{ and } j \in \NN,
 \end{align}
 similar to \eqref{eq:prfThmAContDfEstLip}. Moreover, the uniform convergence implies that $|f_j(0)|\le D$ for some $D>0$ independent of $j\in\NN$.

 Secondly, as $f \equiv \infty$ on $B \cdot \lambda x_0$, $\lambda \in [1-\eta,1]$, the epi-convergence of $f_j$ implies that for every $k \in \NN$ we can choose $j_k \in \NN$ such that
 \begin{align}\label{eq:prfThmACandEstB}
  f_j(x) \geq k \quad \text{ for all }x \in B \cdot \lambda x_0, \lambda\in[1-\eta,1] \text{ and }  j \geq j_k.
 \end{align}

 For the third set $\RR \setminus([-\delta, \delta] \cup B)$, finally, we will just need the general lower bound given by Lemma~\ref{lem:unifLowerBound}, that is, for every $x_0\in \RR^n$ there exists $z \in \RR^n $ and $\beta \in \RR$, such that 
 \begin{align}\label{eq:prfThmACandEstRest}
  f_j(x) \geq \langle z,x\rangle + \beta \quad \text{ for all } x\in \RR^n \text{ with } \|x\|\le \max\{|a|,|b|\}\|x_0\|  \text{ and } j \in \NN. %
 \end{align}

 We are now in position to prove the lower bound of $\Psi(f_j)[x_0]$. Let $k \in \NN$ be arbitrary and fix $j \geq j_k$. If $x_0 \in \RR^n  \backslash \overline{D_{f_j}}$, $\Psi(f_j)[x_0] = \infty$ and there is nothing to prove. Assume in the following that $x_0 \in \overline{D_{f_j}}$. By the convexity of $D_{f_j}$, $\lambda x_0 \in D_{f_j}$ for $\lambda \in (1-\eta,1)$ because $0\in D_f$. Then, by \eqref{eq:prfThmACandEstDelta} applied for $s\lambda$ (as $\lambda < 1$), \eqref{eq:prfThmACandEstB}, \eqref{eq:prfThmACandEstRest}, and the bound on $f_j(0)$, we can estimate
 \begin{align*}
  &\Psi(f_j)[\lambda x_0]=\int_{\RR^\times} \frac{f_j(s\lambda x_0)-f_j(0)}{|s|^2}d\nu(s)\\
&\geq -\int_{[-\delta,\delta]\setminus \{0\}} \!\!C\frac{\lambda\|x_0\|}{|s|} d\nu(s)+\int_B \frac{k-D}{|s|^2}d\nu(s) + \int_{\RR^\times\setminus ([-\delta,\delta]\cup B)} \!\!\!\!\!\!\!\!\!\!\!\!\frac{\langle z,s\lambda x_0\rangle + \beta-D}{|s|^2}d\nu(s).
 \end{align*}
 Note that all terms are well-defined and finite due to the compactness of $\supp \nu$ and the $\nu$-integrability of $|s|^{-1}$, and do not depend on $j \geq j_k$ anymore. Hence, we can rewrite this estimate with constants $C',C'' \in \RR$, not depending on $j$, to obtain
 \begin{align}\label{eq:prfThmaCandEstDfj}
  \Psi(f_j)[\lambda x_0] \geq  C' + C''\lambda + k \int_B \frac{1}{|s|^2}d\nu(s),
 \end{align}
 where $\int_B |s|^{-2}d\nu(s) > 0$ as $\nu(B)>0$ and $|s|^{-2} > 0$. The convexity of $\Psi(f_j)$ together with \eqref{eq:prfThmaCandEstDfj} and the fact that $\Psi(f_j)[0] = 0$ by the definition of $\Psi$ implies
 \begin{align*}
  \lambda \Psi(f_j)[x_0] \geq \Psi(f_j)[\lambda x_0] - (1-\lambda)\Psi(f_j)[0] \geq C' + C''\lambda + k \int_B \frac{1}{|s|^2}d\nu(s).
 \end{align*}
 Letting $\lambda \to 1$, we see that $\Psi(f_j)[x_0] \geq (C'+C'') + k \int_B |s|^{-2}d\nu(s)$ and, consequently, $\Psi(f_j)[x_0] \to \infty$, as $j \to \infty$.
 In total, we have shown that $\Psi(f_j)$ epi-converges to $\Psi(f)$ for $j\rightarrow\infty$, so $\Psi$ is continuous.
 
 It remains to see that $\Psi$ is additive and $\GL(n)$-equivariant. However, the restriction of $\Psi$ to the dense subset $\Conv(\RR^n,\RR)$ is additive and $\GL(n)$-equivariant by construction, so this follows from the continuity of $\Psi$, as the $\GL(n)$-action and the addition of functions in $\ConvO(\RR^n)$ are continuous (see Lemma~\ref{lem:contAddition}).
\end{proof}

\bigskip

Having established the well-definedness of the maps from Theorem~\ref{thm:GLequivCandidate}, we can now show the conditions for their monotonicity and dual translation-invariance claimed in Theorem~\ref{mthm:CharGLEquiv}.

\begin{proposition}\label{prop:GLequivMonDTransInv}
 Suppose that $\nu\in\MeasC^+(\RR)$ with $\int_{\RR^\times}|s|^{-1} d\nu(s) < \infty$ and $c \in \RR$.
 
 Then the map $\Psi: \ConvO(\RR^n) \rightarrow \ConvO(\RR^n)$, defined by \eqref{eq:DefCandidateGLEquiv} is 
 \begin{enumerate}
  \item \label{propGLEquivMon} monotone if and only if $\int_{\RR^\times}|s|^{-2}d\nu(s) \leq c < \infty$;
  \item \label{propGLEquivDTransInv} dually translation-invariant if and only if $\int_{\RR^\times} s^{-1} d\nu(s) = 0$.
 \end{enumerate}
\end{proposition}
\begin{proof}
 Suppose that $\Psi$ is defined by \eqref{eq:DefCandidateGLEquiv} with $\nu$ and $c$ as in the statement of the proposition. If $\Psi$ is monotone its Goodey--Weil distribution is positive by Lemma~\ref{lem:GWdistrMonot}. Let $\phi\in C^\infty(\RR)$ be an even function with $\phi(s)= 0 $ for $|s|<1$, $\phi(s)=1$ for $|s|>2$ and $\phi(s)$ non-decreasing for $1<s<2$, and set $\phi_\varepsilon(y):=\phi\left(\frac{\mathrm{pr}_{\langle \bar e\rangle}y}{\varepsilon}\right)$ for $\varepsilon>0$, $y\in\RR^n$. Then $\phi_\varepsilon\le 1$. Hence, we have $\GW(\Psi_x)(\phi_\varepsilon) \le\GW(\Psi_x)(1)=c$ and thus for $x=\bar e$ and for all $\varepsilon>0$
 \begin{align*}
  \int_{\RR^\times}\frac{\phi\left(\frac{s}{\varepsilon}\right)}{|s|^2}d\nu(s)\le c.
 \end{align*}
 As $\phi$ is non-decreasing on $[0,\infty)$ and non-increasing on $(-\infty,0]$, the integrand converges pointwise monotonously on $\RR^\times$ to $1$. Monotone convergence thus implies
 \begin{align*}
 	\int_{\RR^\times}\frac{1}{|s|^2}d\nu(s)=\lim\limits_{\varepsilon\rightarrow0}\int_{\RR^\times}\frac{\phi\left(\frac{s}{\varepsilon}\right)}{|s|^2}d\nu(s)\le c.
 \end{align*}
If, conversely, $\int_{\RR^\times}|s|^{-2}d\nu(s) \leq c < \infty$, we can split up the integrand to obtain
\begin{align*}
 \Psi(f)[x] = \left(c - \int_{\RR^\times}\frac{1}{|s|^{2}}d\nu(s)\right)f(0) +\int_{\RR^\times} \frac{f(sx)}{|s|^2} d\nu(s),
\end{align*}
for $f \in \Conv(\RR^n,\RR), x \in \RR^n $, which clearly is monotone. This shows claim~\eqref{propGLEquivMon}.

Claim~\eqref{propGLEquivDTransInv} follows directly by plugging in linear functions into \eqref{eq:DefCandidateGLEquiv}.
\end{proof}

We now turn to the ``only if''-part of Theorem~\ref{mthm:CharGLEquiv}, that is, that every continuous, additive and $\GL(n)$-equivariant endomorphism has the form of \eqref{eq:DefCandidateGLEquiv}. We start by analyzing general distributions satisfying the invariance properties that are imposed on the Goodey--Weil-distributions (see Lemma~\ref{lem:equivPropGWdistr}).

\begin{lemma}\label{lem:GLxInvDistr}
 Suppose that $u \in \DistribC(\RR^n)$ and $x \in \RR^n $.
 If $u$ is $\GL(n)_x$-invariant, then
 \begin{itemize}
  \item there exists $c \in \RR$ such that $u(\varphi) = c\varphi(0)$ if $x = 0$,
  \item there exists $u_0 \in \DistribC(\RR)$ such that $u(\varphi) = u_0(s \mapsto \varphi(sx))$ if $x \neq 0$,
 \end{itemize}
 for every $\varphi \in C_c^\infty(\RR^n)$.
\end{lemma}
\begin{proof}
 Assume first that $x = 0$. Then $u$ is $\GL(n)$-invariant and therefore also $\supp u$ is $\GL(n)$-invariant. As $\supp u$ is compact, we conclude that $\supp u \subseteq \{0\}$. By \cite{Hoermander2003}*{Thm.~2.3.4}, there exist constants $a_\alpha \in \RR$, $|\alpha|\leq k$, where $\alpha$ is a multi-index and $k$ is the (finite) order of $u$, such that
 \begin{align*}
  u(\varphi) = \sum_{|\alpha|\leq k} a_\alpha \partial_\alpha \varphi(0), \quad \varphi \in C_c^\infty(\RR^n).
 \end{align*}
 Now let $\lambda > 0$ and consider $\eta_\lambda \in \GL(n)$ defined by $\eta_\lambda(y)=\lambda y$. Then $u(\eta_\lambda \varphi) = u(\varphi)$ by the $\GL(n)$-invariance of $u$, and consequently, by the chain rule,
 \begin{align*}
  \sum_{|\alpha|\leq k} \lambda^{|\alpha|} a_\alpha \partial_\alpha \varphi(0) =u(\eta_\lambda \varphi) = u(\varphi)= \sum_{|\alpha|\leq k} a_\alpha \partial_\alpha \varphi(0), \quad  \varphi \in C_c^\infty(\RR^n).
 \end{align*}
 Using that $u$ has compact support, we plug in the monomials $\psi_{\alpha}(y) = y_1^{\alpha_1}\cdots y_n^{\alpha_n}$, $\alpha = (\alpha_1, \dots, \alpha_n)$, which satisfy $\partial_{\alpha'} \psi_\alpha(0) = \delta_{\alpha, \alpha'}$, to obtain $\lambda^{|\alpha|} = 1$ for every $|\alpha| \leq k$. As this can only be true for every $\lambda>0$ if $|\alpha|=0$, we have shown the first claim.

 Assume now that $x \neq 0$. Then $\GL(n)_x$ consists of all non-degenerate linear maps that keep the span $\langle x\rangle$ of $x$ fixed. As $\supp u$ is compact and $\GL(n)_x$-invariant, we conclude that $\supp u \subseteq \langle x\rangle$. Applying \cite{Hoermander2003}*{Thm.~2.3.5} to the splitting of variables induced by $\RR^n  = \langle x \rangle \oplus \langle x \rangle^\perp$, we obtain distributions $u_\alpha \in \DistribC(\RR)$, $|\alpha| \leq k$ and $\alpha=(0,\alpha')$, of order $k-|\alpha|$, such that
 \begin{align*}
  u(\varphi) = \sum_{|\alpha| \leq k} u_\alpha(\varphi_\alpha), \quad \varphi \in C_c^\infty(\RR^n),
 \end{align*}
 where $\varphi_{(0,\alpha')}(t) = \partial_{\alpha'}\varphi(tx,y')|_{y'=0}$. For $\lambda > 0$ and taking $\tau_\lambda \in \GL(n)_x$ defined by $\tau_\lambda(tx,y') = (tx, \lambda y')$, we conclude as before that
 \begin{align*}
  \sum_{|\alpha| \leq k} \lambda^{|\alpha|} u_\alpha(\varphi_\alpha) = \sum_{|\alpha| \leq k} u_\alpha(\varphi_\alpha), \quad \varphi \in C_c^\infty(\RR^n),
 \end{align*}
 and, by taking $\varphi(tx,y') = \tilde{\varphi}(t)\psi_{\alpha'}(y')$ for $\tilde{\varphi} \in C_c^\infty(\RR)$ and the monomial $\psi_{\alpha'}$ from the first case, we see that $\lambda^{|\alpha|} = 1$ for all $|\alpha| \leq k$, so $u_\alpha = 0$ for $|\alpha| > 0$ as before. This completes the proof of the second claim.
\end{proof}

\noindent
We are now in position to complete the

\begin{proof}[Proof of Theorem~\ref{mthm:CharGLEquiv}]
 By Theorem~\ref{thm:GLequivCandidate} and Proposition~\ref{prop:GLequivMonDTransInv}, we are left to prove that every continuous, additive and $\GL(n)$-equivariant map $\Psi:\ConvO(\RR^n) \rightarrow \ConvO(\RR^n)$ is of the form~\eqref{eq:DefCandidateGLEquiv}.
 
 First, note that $\Psi$ restricts to an endomorphism of $\Conv(\RR^n,\RR)$ by Lemma~\ref{lem:restrictRadEquiv}. We can therefore consider its family of Goodey--Weil distributions $(\GW(\Psi_x))_{x \in \RR^n }$. Every $\GW(\Psi_x)$ is invariant under the stabilizer $\GL(n)_x$ of $x \in \RR^n $ by Lemma~\ref{lem:equivPropGWdistr}. Applying Lemma~\ref{lem:GLxInvDistr}, there is $c \in \RR$ and a distribution $u \in \DistribC(\RR)$ such that $\Psi(f)[0] = cf(0)$ and
 \begin{align*}
  \GW(\Psi_{\bar e})(f) = u(s \mapsto f(s \bar e)), \quad f \in \Conv(\RR^n,\RR) \cap C^\infty(\RR^n),
 \end{align*}
 where $\bar e$ is a pole of $\unitsurfn$ (but could be some arbitrary, non-zero element of $\RR^n$). The $\GL(n)$-equivariance of $\Psi$ implies that for $\eta_x \in \GL(n)$, $x \in \RR^n \setminus \{0\}$, with $\eta_x \bar e = x$,
 \begin{align*}
  \Psi(f)[x] = \Psi(f)[\eta_x \bar e] = \Psi(\eta_x^{-1} \cdot f)[\bar e] = u(s \mapsto (\eta_x^{-1} \cdot f)(s \bar e)) = u(s \mapsto f(s x))
 \end{align*}
 for every $f \in \Conv(\RR^n,\RR)\cap C^\infty(\RR^n)$. Note that $\Psi(f)[x]$ depends only on the restriction of $f$ to $\langle x \rangle$. Conversely, every $f \in \Conv(\RR,\RR)$ defines a convex function $\tilde{f}$ on $\RR^n $ by $\tilde{f}(y) = f(\bar e^\ast(y))$, where $\bar e^\ast\in (\RR^n)^*$ is invariant under the stabilizer $\GL(n)_{\bar e}$ of $\bar e$ and satisfies $\bar{e}^\ast(\bar e)=1$. We can therefore define an endomorphism $\Psi^\RR: \Conv(\RR,\RR) \rightarrow \Conv(\RR, \RR)$ by $\Psi^\RR(f)[t] = \Psi(\tilde{f})[t\bar e]$, $t \in \RR$. If $f$ is smooth, then so is $\tilde{f}$, and $\Psi^\RR (f)$ is given by
 \begin{align}\label{eq:prfThmADefPsiRR}
  \Psi^\RR(f)[t] = \Psi(\tilde{f})[t\bar e] = u(s \mapsto f(st)), \quad t \in \RR.
 \end{align}
 Observe that $\Psi^\RR$ determines $\Psi$ completely by this equation.
 
 Now let $f\in\Conv(\RR,\RR)$ be a smooth convex function. As $\Psi^\RR(f)[t]$ is convex, the second derivative of $\Psi^\RR(f)[t]$ is non-negative, and using \cite{Hoermander2003}*{Thm.~2.1.3} we obtain
 \begin{align}\label{eq:prfThmAPosSGW}
  0 \leq \partial^2_t \Psi^\RR(f)[t] = \partial^2_t u(s \mapsto f(s t)) = u (s \mapsto s^2 f''(st)).
 \end{align}
 If $\phi \in C_c^\infty(\RR)$ is non-negative, let $\Phi \in \Conv(\RR,\RR)$ be such that $\Phi''(s)=\phi(s) \geq 0$. Then the inequality in \eqref{eq:prfThmAPosSGW} yields for $f = \Phi$
 \begin{align*}
  (s^2 u)(\phi) = u \left(s \mapsto s^2\phi(s)\right)=u\left(s \mapsto s^2\Phi''(s)\right) = \partial^2_{t} \Psi^\RR(\Phi)[t]|_{t=1} \geq 0,
 \end{align*}
 that is, $s^2 u$ is a non-negative distribution (with compact support). We can therefore find a non-negative measure $\nu \in \MeasC^+(\RR)$ such that $s^2 u = \nu$ as distributions.

 Next, note that if $\phi \in C^\infty(\RR)$ vanishes on a neighborhood of $0$, then 
 \begin{align}\label{eq:prfThmARepVanAround0}
  u(\phi) = u\left(s \mapsto s^2 \frac{\phi(s)}{|s|^2}\right) = \int_{\RR^\times} \frac{\phi(s)}{|s|^2}d\nu(s).
 \end{align}
 Hence, by approximating the (convex) functions $f_\delta(t) = (t-\delta)_+ + (\delta-t)_+$,  where $t_+ = \max\{t,0\}$, by smooth functions, we get
 \begin{align}\label{eq:prfThmAIntegNu}
  \Psi^\RR(f_\delta)[1] = \int_{\RR\setminus [-\delta,\delta]} \frac{f_\delta(s)}{|s|^2}d\nu(s) = \int_{\RR\setminus [-\delta,\delta]} \frac{|s|-\delta}{|s|^2}d\nu(s).
 \end{align}
 For $\delta \to 0$, $f_\delta$ epi-converges to $f(t) = |t|$, so by continuity, the left-hand side of \eqref{eq:prfThmAIntegNu} converges to $\Psi^\RR(f)[1]$. By monotone convergence, the right-handside of $\eqref{eq:prfThmAIntegNu}$ converges to $\int_{\RR^\times}|s|^{-1} d\nu$, which thus must be finite.
 
 \bigskip
 
 Hence, we can apply Theorem~\ref{thm:GLequivCandidate} to $\nu$ and $c$ to obtain a continuous, additive map $\hat\Psi:\ConvO(\RR^n) \rightarrow \ConvO(\RR^n)$. We claim that $\Psi=\hat{\Psi}$. It is enough to show that their Goodey--Weil distributions coincide or, equivalently, $\Psi^\RR (f)[1] = \hat\Psi^\RR(f)[1]$ for every $f \in \Conv(\RR,\RR)$, where $\hat\Psi^\RR$ is defined similarly to \eqref{eq:prfThmADefPsiRR}. Indeed, as $\hat\Psi$ satisfies the conditions of the theorem, we may repeat the steps taken for $\Psi$ to conclude that $\hat\Psi$ is uniquely determined by $\hat\Psi^\RR$ given by
 \begin{align*}
  \hat\Psi^\RR(f)[t] = cf(0)+\int_{\RR^\times}\frac{f(st)-f(0)}{|s|^2}d\nu(s),\quad t \in \RR,
 \end{align*}
 for every $f \in \Conv(\RR,\RR)$, and thus, by $\GL(n)$-equivariance, just by $\hat\Psi^\RR(f)[1]$.
 
 In order to show $\Psi^\RR = \hat\Psi^\RR$, first let $f\in\Conv(\RR,\RR)$ be a smooth convex function that vanishes in a neighborhood of $0$. Then \eqref{eq:prfThmARepVanAround0} shows that
 \begin{align*}
  \Psi^\RR(f)[1]=\GW(\Psi^\RR_1)[f]=u(f)=\hat\Psi^\RR(f)[1].
 \end{align*}
 If $f\in\Conv(\RR,\RR)$ is an arbitrary convex function that vanishes in a neighborhood of $0$, we can use a mollifier and approximate $f$ by a sequence of smooth convex functions with the same property, so $\Psi^\RR(f)[1]=\hat\Psi^\RR(f)[1]$, as $\Psi^\RR$ and $\hat\Psi^\RR$ are both continuous. Next, let $f\in\Conv(\RR,\RR)$ satisfy $f(0)=0\le f(t)$ for all $t\in\RR$. Then $f_\delta(t):=\max(f(t)-\delta,0)$ vanishes on a neighborhood of $0$ and converges to $f$ for $\delta\rightarrow0$. Thus $\Psi^\RR$ and $\hat\Psi^\RR$ coincide on functions of this type by continuity.
 
 \bigskip
 
 \noindent Next, note that, by $\GL(n)$-equivariance, $\Psi^\RR(s \mapsto 1)$ is constant, and therefore
 \begin{align*}
 \Psi^\RR(s \mapsto 1)[1]=\Psi^\RR(s \mapsto 1)[0] = c = \hat\Psi^\RR(s \mapsto 1)[1]. 
 \end{align*}
 By approximating the function $g_+(t) = t_+$ by the functions $g_{+,\delta}(t) = (t-\delta)_+$, $\delta \to 0$, we further see that $\Psi^\RR(g_+) = \hat\Psi^\RR(g_+)$ and similarly $\Psi^\RR(g_-)=\hat\Psi^\RR(g_-)$ for $g_-(t) = (-t)_+$. If we consider $g(t)=t$, we thus obtain 
 \begin{align*}
 	\Psi^\RR(g_-)\!+\!\Psi^\RR(g)=\Psi^\RR(g_-\! +\! g)=\Psi^\RR(g_+)=\hat\Psi^\RR(g_+)=\hat\Psi^\RR(g_-\! +\! g)= \hat\Psi^\RR(g_-)\!+\!\hat\Psi^\RR (g).
 \end{align*}
 As $\Psi^\RR(g_-) = \hat\Psi^\RR(g_-)$ and because both sides are finite, we see that $\Psi^\RR$ and $\hat\Psi^\RR$ coincide on linear functions. 
 
 Finally, let $f\in\Conv(\RR,\RR)$ be an arbitrary convex function. Then there exists an affine function $g$ on $\RR$ such that $\hat{f}(t):=f(t)-g(t)$ is a convex function with $\hat f(t)\ge \hat f(0)=0$. By additivity and the previously shown,
 \begin{align*}
 	\Psi^\RR(f)=\Psi^\RR(\hat{f})+\Psi^\RR(g)=\hat\Psi^\RR(\hat{f})+\hat\Psi^\RR(g)=\hat\Psi^\RR(f),
 \end{align*}
 which finishes the proof.
 \end{proof}

\noindent
In the remainder of this section, we give a proof of Theorem~\ref{mcor:CharGLEquivWholeConv}, which makes use of Theorem~\ref{mthm:CharGLEquiv}.

\begin{proof}[Proof of Theorem~\ref{mcor:CharGLEquivWholeConv}]
	First note that every map $\Psi:\Conv(\RR^n)\rightarrow\Conv(\RR^n)$ defined by \eqref{eq:MainCorGLEquivWholeConv} or by $\Psi \equiv 0, \mathbbm{1}_{\{0\}}^\infty$ is continuous, additive and $\GL(n)$-equivariant.
	
	Now let $\Psi:\Conv(\RR^n) \rightarrow \Conv(\RR^n)$ be non-trivial, that is, $\Psi \not \equiv 0$, as well as continuous, additive and $\GL(n)$-equivariant, and consider $\Psi(0) = \Psi(y \mapsto 0)$, which is $\GL(n)$-invariant by the $\GL(n)$-equivariance of $\Psi$. We claim that $\Psi(0) \in \Conv(\RR^n,\RR)$ or $\Psi(0) = \mathbbm{1}_{\{0\}}^\infty$.
	
	Indeed, if there exists $x \in \RR^n \setminus\{0\}$ with $\Psi(0)[x] < \infty$, then  $\Psi(0) < \infty$ on $\RR^n \setminus\{0\}$, since $\Psi(0)$ is $\GL(n)$-invariant, and thus $\Psi(0)[0]<\infty$ by convexity. Moreover, if $\Psi(0)[x] = \infty$ for some $x \in \RR^n  \backslash \{0\}$, then  $\Psi(0)[x] =\infty$ for every $x \in \RR^n  \setminus\{0\}$ by the $\GL(n)$-invariance. The additivity of $\Psi$ and the fact that $\Psi(0)[0] < \infty$, as $\Psi(0) \in \Conv(\RR^n)$, then imply that $\Psi(0)[0] = \Psi(0)[0]+\Psi(0)[0]$, that is, $\Psi(0)[0] = 0$ and so $\Psi(0) = \mathbbm{1}_{\{0\}}^\infty$.
	
	\medskip
	
	If $\Psi(0) = \mathbbm{1}_{\{0\}}^\infty$, then $\Psi(f)= \Psi(f+0) = \Psi(f) + \mathbbm{1}_{\{0\}}^\infty$ for all $f\in\Conv(\RR^n,\RR)$, so $\Psi(f)[0]$ must be finite and $\Psi(f)[x]=\infty$ for all $x\in \RR^n \setminus\{0\}$ for any $f\in\Conv(\RR^n,\RR)$. As $\Psi$ is continuous with respect to the topology induced by epi-convergence, this implies that the map $f\mapsto \Psi(f)[0]$ is continuous on $\Conv(\RR^n,\RR)$. We can thus consider its Goodey--Weil distribution, which is $\GL(n)$-invariant. By Lemma~\ref{lem:GLxInvDistr}, there exists $c \in \RR$ such that $\Psi(f)[0] = c f(0)$ for every $f \in \Conv(\RR^n,\RR) \cap C^\infty(\RR^n)$. However, the map $f \mapsto c f(0) + \mathbbm{1}_{\{0\}}^\infty$ does not possess a continuous extension to $\Conv(\RR^n)$, unless $c = 0$, as can be seen by approximating any function $f \in \Conv(\RR^n)$ that is infinite in a neighborhood of $0$ with finite convex functions, hence, $\Psi \equiv \mathbbm{1}_{\{0\}}^\infty$. 
	
	\medskip
	
	If $\Psi(0) \in \Conv(\RR^n,\RR)$, on the other hand, then $\Psi$ maps $\Conv(\RR^n,\RR)$ into itself by the same reasoning as in the proof of Lemma~\ref{lem:restrictRadEquiv}. We can therefore consider the family of Goodey--Weil distributions $(\GW(\Psi_x))_{x\in \RR^n }$ and apply (the arguments of) Theorem~\ref{mthm:CharGLEquiv} to obtain $c \in \RR$ and $\nu \in \MeasC^+(\RR)$, such that %
	\begin{align}\label{eq:prfCorB}
		\Psi(f)[x] = c f(0) + \int_{\RR^\times} \frac{f(sx)-f(0)}{|s|^2} d\nu(s), \quad x \in \RR^n ,
	\end{align}
	for all $f \in \Conv(\RR^n,\RR)$. We want to use \eqref{eq:prfCorB} to show that $\Psi$ is of the claimed form.
	
	\medskip 
	
	First fix $y_0\in \RR^n \setminus\{0\}$. Then there exists $x_0\in \RR^n $ such that $\Psi(\mathbbm{1}_{\{y_0\}}^\infty)[x_0]<\infty$. We claim that this implies $\supp\Psi_{x_0}\subseteq\{y_0\}$. To see this, let $\phi\in C^\infty_c(\RR^n)$ be a function with $y_0\notin \supp\phi$. As $\supp \phi$ is closed, we can choose $\delta>0$ such that $B_{\delta}(y_0) \cap \supp\phi=\emptyset$. Now consider the function $f(y):=\|y-y_0\|^2$. Then there exists $\varepsilon>0$ such that $f+t\phi$ is convex for all $t\in(-\varepsilon,\varepsilon)$ and for such $t$ we can use \eqref{eq:GWonCinfbyPsi} to calculate
	\begin{align*}
		t\GW(\Psi_{x_0})[\phi]=&\Psi(f+t\phi)[x_0]-\Psi(f)[x_0]\\
		=&\Psi(f+t\phi)[x_0]+\Psi(\mathbbm{1}_{{\{y_0\}}}^\infty)[x_0]-\left(\Psi(f)[x_0]+\Psi(\mathbbm{1}_{{\{y_0\}}}^\infty)\right)[x_0]\\
		=&\Psi(\mathbbm{1}_{{\{y_0\}}}^\infty+f+t\phi)[x_0]-\Psi(\mathbbm{1}_{{\{y_0\}}}^\infty+f)[x_0]\\
		=&\Psi(\mathbbm{1}_{{\{y_0\}}}^\infty)[x_0]-\Psi(\mathbbm{1}_{{\{y_0\}}}^\infty)[x_0]=0.
	\end{align*}
	Here we have used $\Psi(\mathbbm{1}_{\{y_0\}}^\infty)[x_0]<\infty$, the additivity of $\Psi$, and the fact that $f$ and $f+t\phi$ vanish in $y_0$.
	Thus, $\supp\GW(\Psi_{x_0})\subset\{y_0\}$.
	
	\medskip
	
	Next, we want to show that for every $y _0 \in \RR^n \setminus\{0\}$ we can choose $x_0 \neq 0$ with $\Psi(\mathbbm{1}_{\{y_0\}}^\infty)[x_0]<\infty$. Note that this holds for all $y_0\in \RR^n\setminus\{0\}$ as soon as it holds for one point due to the $\GL(n)$-equivariance of $\Psi$. Assume for the sake of a contradiction that $\dom\Psi(\mathbbm{1}_{{\{y_0\}}}^\infty) \subseteq \{0\}$ for all $y_0\in \RR^n \setminus\{0\}$. Then $\Psi(\mathbbm{1}_{{\{y_0\}}}^\infty)[0]<\infty$, as $\Psi(\mathbbm{1}_{{\{y_0\}}}^\infty)$ is not identical to $+\infty$. 
	
	As the maps $y\mapsto j\|y-y_0\|$ epi-converge to $\mathbbm{1}_{\{y_0\}}^\infty$, the continuity of $\Psi$ implies that $\Psi(j\|\cdot-y_0\|) \rightarrow \Psi(\mathbbm{1}_{{\{y_0\}}}^\infty)$ as $j\rightarrow\infty$. By the definition of epi-converge, point~\ref{defEpiConvConvSeq}, there exists a sequence $(x_j)_{j}$ in $\RR^n $ converging to $0$ such that $\Psi(j\|\cdot-y_0\|)[x_j]$ converges to $\Psi(\mathbbm{1}_{\{y_0\}}^\infty)[0]$ for $j\rightarrow\infty$. In particular, \eqref{eq:prfCorB} implies that
	\begin{align}\label{eq:prfCorBApprxInd}
		\Psi(j\|\cdot-y_0\|)[x_j] = j \left[c \|y_0\| + \int_{\RR^\times} \frac{\|sx_j-y_0\|-\|y_0\|}{|s|^2} d\nu(s)\right]
	\end{align}
	converges to $\Psi(\mathbbm{1}_{\{y_0\}}^\infty)[0]<\infty$ for $j\rightarrow\infty$, which is only possible if the expression in the brackets converges to $0$. Since $\|y_0\| \neq 0$ and the integrand is bounded by 
	\begin{align*}
	\left|\frac{\|sx_j-y_0\|-\|y_0\|}{|s|^2}\right|\le \frac{\|x_j\|}{|s|}, \quad s \in \RR^\times,
	\end{align*}
	with $\|x_j\|\to0$ and $|s|^{-1}$ being $\nu$-integrable, the expression in the bracket in \eqref{eq:prfCorBApprxInd} can only converge to $0$ if $c=0$.
	
	Next, let $t \in \RR^\times$ and consider the constant sequence $\tilde x_j = ty_0$. Again by the definition of epi-convergence, point~\ref{defEpiConvInfBed}, 
	\begin{align*}
		\Psi(\mathbbm{1}_{\{y_0\}}^\infty)[ty_0]\le \liminf_{j\rightarrow\infty} \Psi(j\|\cdot-y_0\|)[\tilde x_j]=\|y_0\|\liminf_{j\rightarrow\infty}j\int_{\RR^\times} \frac{|st-1|-1}{|s|^2} d\nu(s),
	\end{align*}
	and since $\Psi(\mathbbm{1}_{\{y_0\}}^\infty)[ty_0] = \infty$ for $t \neq 0$, we must have
	\begin{align*}
	 \int_{\RR^\times} \frac{|st-1|-1}{|s|^2} d\nu(s) > 0.
	\end{align*}
    Choosing $t \neq 0$ such that $|s|\cdot|t| \leq 1$ for all $s \in \supp \nu$ (using that $\supp \nu$ is compact), we obtain
    \begin{align*}
     0 < \int_{\RR^\times} \frac{|st-1|-1}{|s|^2} d\nu(s) = -t \int_{\RR^\times} \frac{1}{s} d\nu(s),
    \end{align*}
    which yields a contradiction, as we can flip the sign of $t$. We therefore conclude that for every $y_0 \in \RR^n \setminus\{0\}$ the set $\dom\Psi(\mathbbm{1}_{\{y_0\}}^\infty)$ contains a point $x_0$ distinct from $0$.

    By the first step, $\supp \GW(\Psi_{x_0}) \subseteq \{y_0\}$. If $\Psi_{x_0}=0$ for some $x_0\ne 0$, then $\Psi_x=0$ for all $x \in \RR^n \setminus\{0\}$ by $\GL(n)$-equivariance, and, by convexity, $\Psi_0=0$, so $\Psi \equiv 0$, which contradicts the assumption of non-triviality. Hence, $\Psi_{x_0}\ne 0$, and since $\supp \GW(\Psi_{x_0})=\{y_0\}$ is $\GL(n)_{x_0}$-invariant, $y_0 \in \langle x_0\rangle$, that is, there exists $\mu \in \RR^\times$ such that $y_0 = \mu x_0$. Note that the $\GL(n)$-equivariance of $\Psi$ implies that $\mu$ does not depend on $y_0 \in \RR^n \setminus\{0\}$. Consequently, $\supp \GW(\Psi_{x_0}) = \{\mu x_0\}$ for all $x_0 \in \RR^n \setminus\{0\}$.
    
    Finally, using the arguments in the proof of Theorem~\ref{mthm:CharGLEquiv}, $\nu = s^2 u$, where $u \in \DistribC(\RR)$ is given by $\GW(\Psi_{\bar e})[\phi] = u(s \mapsto \phi(s\bar e)), \phi \in C_c^\infty(\RR^n)$, and we conclude $\emptyset \neq \supp \nu \subseteq \supp u= \{\mu\}$. By \eqref{eq:prfCorB} and the fact that $\Psi(f)[x_0]$ does not depend on $f(0)$ anymore, it is then easy to see that for $\lambda=\frac{\nu(\{\mu\})}{\mu^2}>0$ we have
	\begin{align*}
		\Psi(f)[x_0]=\lambda f(\mu x_0),\quad f\in\Conv(\RR^n,\RR),
	\end{align*}
	which extends to $\Conv(\RR^n)$ by continuity.
\end{proof}

\section{Proof of Theorem~\ref{mthm:CharMonRadEquiv} and Corollary~\ref{mcor:CharRadiallyEquivWholeConv}}

In this section, we prove Theorem~\ref{mthm:CharMonRadEquiv} and show that it implies Corollary~\ref{mcor:CharRadiallyEquivWholeConv}. As in the $\GL(n)$-equivariant setting, we first show that \eqref{eq:MainThmRadiallyEquiv} well-defines an endomorphism with the claimed properties. Note that the arguments in the proof of the next theorem resemble in big parts arguments in the proof of \cite{Hofstaetter2021}*{Prop.~3.3}, which corresponds to the case $\supp \mu \subseteq \unitsurfn$. However, we do not know whether one can use \cite{Hofstaetter2021}*{Prop.~3.3} directly.

\begin{theorem}\label{thm:MonRadCandidate}
 Suppose that $\mu\in\MeasC^+(\RR^n)$ is an $\SO(n-1)$-invariant measure.
 
 \noindent
 Then the map $\Psi: \ConvO(\RR^n) \rightarrow \ConvO(\RR^n)$ defined by
 \begin{align}\label{eq:exMonRadEq}
  \Psi(f)[x] = \int_{\RR^n } f(\|x\|\vartheta_x y) d\mu(y), \quad x \in \RR^n \backslash\{0\},
 \end{align}
 and $\Psi(f)[0] = \liminf_{\|x\|\to 0} \Psi(f)[x]$ for every $f \in \ConvO(\RR^n)$ is continuous, additive, monotone, radially and $\SO(n)$-equivariant. Moreover, $\Psi$ is dually translation-invariant if and only if $\int_{\RR^n} y d\mu(y) = 0$.
\end{theorem}
\begin{proof}
 Note that, if well-defined, $\Psi$ is clearly an additive, monotone, radially and $\SO(n)$-equivariant mapping. Moreover, $\Psi$ is dually translation-invariant if and only if $\Psi(\langle a, \cdot\rangle) = 0$ for every $a \in \RR^n$, which is equivalent to $\int_{\RR^n} y d\mu(y) = 0$. We therefore need to show that $\Psi(f)$ is a well-defined function in $\ConvO(\RR^n)$ for every $f \in \ConvO(\RR^n)$ and that the mapping $\Psi:\ConvO(\RR^n) \rightarrow \ConvO(\RR^n)$ is continuous.
 
 \medskip
 
First let $f \in \ConvO(\RR^n)$ and note that the lower semi-continuity of $f$ implies that $f$ is bounded from below on every compact subset of $\RR^n $. As $\supp \mu$ is compact by assumption,  the map $y \mapsto f(\|x\|\vartheta_x y)$ is therefore bounded from below for $y \in \supp \mu$ for every fixed $x \in \RR^n \setminus\{0\}$ , and thus the integral in \eqref{eq:exMonRadEq} is well-defined for every $x \in \RR^n  \setminus \{0\}$, taking values in $(-\infty,\infty]$. Moreover, \eqref{eq:exMonRadEq} does not depend on the choice of $\vartheta_x$, as $\mu$ is $\SO(n-1)$-invariant.
 
 Noting further that $\Psi(f)[x]$ only depends on the values of $f$ on the (compact) set $\|x\|\vartheta_x \supp \mu$, we conclude that $\Psi(f)[x]$ is finite whenever $\|x\|\vartheta_x \supp \mu$ is contained in $\interior \dom f$. Indeed, the continuity of $f$ on $\interior \dom f$ implies that $f$ is bounded on the domain of integration of \eqref{eq:exMonRadEq} and thus $\Psi(f)[x] < \infty$. Since $0 \in \interior \dom f$ and $\supp \mu$ is compact, $\Psi(f)$ is finite on a neighborhood of the origin (and in particular proper).
 
 \medskip
 
 Next, to prove that $\Psi(f)$ is convex, we can assume without loss of generality that $f$ is finite. Indeed, for every $f \in \ConvO(\RR^n)$ there is a sequence $f_j \in \Conv(\RR^n,\RR)$ such that $(f_j(x))_{j}$ is monotone for every $x \in \RR^n $ and such that $f_j$ epi-converges to $f$ (take, e.g., the Moreau envelope of $f$, see \cite{Rockafellar1998}*{Ch.~1.G}). Monotone convergence of the integral in \eqref{eq:exMonRadEq} then implies that
 \begin{align*}
  \Psi(f)[x] = \sup_{j \in \NN} \Psi(f_j)[x], \quad x \in \RR^n  \setminus\{0\},
 \end{align*}
 and, hence, $\Psi(f)$ is convex on $\RR^n  \setminus\{0\}$ as supremum of convex functions. Moreover, we only need to show convexity on $\RR^n  \setminus\{0\}$ since $\Psi(f)$ must then be finite and convex and therefore continuous on $B_\delta(0) \setminus \{0\}$ for some $\delta > 0$. It is easy to check that the extension of $\Psi(f)$ to $0$ by lower semi-continuity is again convex (and therefore the extension is continuous on $B_\delta$).
 
 In order to prove convexity of $\Psi(f)$ for $f \in \Conv(\RR^n,\RR)$, we first assume that $\mu$ is absolutely continuous with respect to the Lebesgue measure on $\RR^n $ with density $\rho$. Rewriting the integral in \eqref{eq:exMonRadEq} using polar coordinates yields for $x \in \RR^n  \setminus \{0\}$
 \begin{align*}
  \Psi(f)[x] = \int_{\RR^n} f(\|x\| \vartheta_x y) \rho(y) dy = \int_0^\infty \int_\unitsurfn f(\|x\|\vartheta_x r u) \rho(ru) du \, r^{n-1} dr,
 \end{align*}
 where we denote by $du$ the spherical Lebesgue measure on $\unitsurfn$. Note that, by compactness, $\supp \mu \subseteq B_R(0)$ for some $R>0$, and, hence, $\rho(ru) \equiv 0$ for $r>R$, $u \in \unitsurfn$. Letting $\nu_r$, $r \in (0,R]$, be the measure on $\unitsurfn$ with density $\rho(r\cdot)$ with respect to the spherical Lebesgue measure, we can use Theorem~\ref{thm:HSKlasseAsplEndos} to identify
 \begin{align*}
  \Psi(f)[x] = \int_0^R \Psi_{\nu_r}(f)[rx] r^{n-1} dr, \quad x\in \RR^n  \setminus \{0\},
 \end{align*}
 where $\Psi_{\nu_r}(f)$ is convex for every $r \in (0,R]$ by Theorem~\ref{thm:HSKlasseAsplEndos}. From this representation it is clear that $\Psi(f)$ is convex on $\RR^n  \setminus\{0\}$.
 
 For a general measure $\mu$ (that might not be absolutely continuous with respect to the Lebesgue measure), we use a mollification of $\mu$. Let $\phi:\RR^n  \rightarrow [0,\infty)$ be a smooth, radially symmetric function with $\supp\phi \subseteq B_1(0)$ and $\int_{\RR^n} \phi dx = 1$. Denote $\phi_\varepsilon(x) = \frac{1}{\varepsilon^n}\phi(\frac{x}{\varepsilon})$, $\varepsilon > 0$, and consider
 \begin{align*}
  (\mu \ast \phi_\varepsilon)(y) = \int_{\RR^n} \phi_\varepsilon(y-z) d\mu(z), \quad y \in \RR^n ,
 \end{align*}
 which is a smooth, non-negative function with compact support (see, e.g., \cite{Hoermander2003}*{Thm.~4.1.1}). Moreover, the radial symmetry of $\phi$ implies that $\mu \ast \phi_\varepsilon$ is $\SO(n-1)$-invariant. We can therefore use what we have proved before to see that
 \begin{align*}
  \Psi_\varepsilon(f)[x] = \int_{\RR^n} f(\|x\|\vartheta_x y) (\mu \ast \phi_\varepsilon)(y) \,dy, \quad x \in \RR^n  \setminus\{0\},
 \end{align*}
 is convex for every $f \in \Conv(\RR^n,\RR)$. Since by \cite{Hoermander2003}*{Thm.~4.1.4}, $\mu \ast \phi_\varepsilon$ converges weakly to $\mu$ for $\varepsilon \to 0^+$, $\Psi_\varepsilon(f)[x]$ converges to $\Psi(f)[x]$ for every $x \in \RR^n \setminus\{0\}$. We conclude that $\Psi(f)$ is convex on $\RR^n  \setminus\{0\}$ as pointwise limit of convex functions and, hence, as noted before, on $\RR^n $.

 Since the lower semi-continuity of $\Psi(f)$, $f \in \ConvO(\RR^n)$, is a direct consequence of Fatou's lemma (see also \cite{Hofstaetter2021}*{Prop.~3.3}), we conclude that $\Psi(f) \in \ConvO(\RR^n)$ for every $f \in \ConvO(\RR^n)$.
 
 \medskip

 It remains to prove the continuity of $\Psi$ on $\ConvO(\RR^n)$. Let $(f_j)_{j} \subseteq \ConvO(\RR^n)$ be a sequence epi-converging to $f \in \ConvO(\RR^n)$ and consider the set
 \begin{align*}
 	D_f:=\{x\in \RR^n : \|x\| \vartheta_x\supp\mu\subset \interior\dom f \text{ for some }\vartheta_x\in \SO(n)\text{ with }\|x\|\vartheta_x \bar e=x\}.
 \end{align*} 
 Note that $D_f$ contains an open ball around the origin, as $\interior\dom f$ is a neighborhood of the origin and $\supp\mu$ is compact. We will show that $\Psi(f_j)$ converges pointwise to $\Psi(f)$ on the dense set $(D_f\setminus\{0\})\cup \RR^n \setminus\overline{D_f}$, which by Lemma~\ref{lem:epiConvEquiv}(\ref{lem:epiConvEquivPWonDense}) implies that $(\Psi(f_j))_j$ epi-converges to $\Psi(f)$.
 
 \medskip
 
 First, let $x\in D_f\setminus\{0\}$ be given. By the compactness of $\supp \mu$, there exists $\varepsilon>0$ such that $A_\varepsilon:=\|x\|\vartheta_x\supp\mu +B_{\varepsilon}(0)\subset\interior\dom f$. As $A_\varepsilon$ is compact and does not intersect the boundary of $\dom f$, Lemma~\ref{lem:epiConvEquiv}(\ref{lem:epiConvEquivUnifComp}) implies that $(f_j)_j$ converges uniformly to $f$ on $A_\varepsilon$. In particular, there exists $j_0\in\NN$ such that $f_j<\infty$ on $A_\varepsilon$ for all $j\ge j_0$. Using the uniform convergence of $(f_j)_j$ and that $\supp\mu$ is compact, dominated convergence implies $\lim\limits_{j\rightarrow\infty}\Psi(f_j)[x]=\Psi(f)[x]$.
 
 \medskip
 
 Next, let $x\in \RR^n \setminus\overline{D_f}$ be given and choose $\delta>0$ such that $B_\delta(x) \subseteq \RR^n \setminus\overline{D_f}$. Taking $\widetilde x = \lambda x$ with $\lambda = 1-\frac{\delta}{2\|x\|} < 1$, we have $\widetilde x \in \RR^n \setminus\overline{D_f}$, so by assumption there exists $y \in (\|\widetilde x\|\vartheta_{\widetilde x} \supp \mu) \cap (\RR^n \setminus\interior\dom f)$. As $\interior\dom f$ is a convex neighborhood of $0$, the ray through $y$ emanating from $0$ intersects the boundary of $\dom f$ in at most one point, which must be contained in the segment $(0,y]$. Hence, $\lambda^{-1} y \in \RR^n \setminus\overline{\dom f}$, as $\lambda < 1$, and there exists a compact neighborhood $A \subseteq \RR^n  \setminus\overline{\dom f}$ of $\lambda^{-1}y$.
 
 Noting that we can choose $\vartheta_x=\vartheta_{\widetilde x}$  since $x$ and $\widetilde x$ are collinear, we also have $\lambda^{-1} y \in \|x\| \vartheta_x \supp \mu$, and, consequently, $\mu(\widetilde A)>0$, where for simplicity we denote $\widetilde A = \frac{1}{\|x\|}\vartheta_x^{-1} A$. As $f$ is bounded from below on $\|x\|\vartheta_x\supp \mu$ by semi-continuity and $f\equiv+\infty$ on $A$, this directly implies
 \begin{align*}
  \Psi(f)[x] = \int_{\RR^n} f(\|x\|\vartheta_x y) d\mu(y) = \infty.
 \end{align*}
 Moreover, $A$ does not intersect the boundary of $\dom f$, so the sequence $(f_j)_j$ converges uniformly to $f$ on $A$ by Lemma~\ref{lem:epiConvEquiv}(\ref{lem:epiConvEquivUnifComp}). As $f\equiv +\infty$ on $A$, this implies that for every $k\in\mathbb{N}$ there exists $j_k\in \NN$ such that $f_j\ge k$ on $A$ for all $j\ge j_k$. By Lemma~\ref{lem:unifLowerBound}, there further exists an affine function $g:\RR^n \rightarrow\RR$ such that $f_j\ge g$ on the compact set $\|x\|\vartheta_x\supp\mu$. Consequently, we can estimate for $j\ge j_k$
 \begin{align*}
 \Psi(f_j)[x]=&\int_{\RR^n }f_j(\|x\|\vartheta_x y)d\mu(y)=\int_{\RR^n \setminus \widetilde A}f_j(\|x\|\vartheta_x y)d\mu(y)+\int_{\widetilde A}f_j(\|x\|\vartheta_x y)d\mu(y)\\
 \ge&\int_{\RR^n \setminus \widetilde A}g(\|x\|\vartheta_x y)d\mu(y)+k\int_{ \widetilde A}d\mu(y)=C+k\mu(\widetilde A)
 \end{align*}
 for some constant $C$ independent of $j,k\in\mathbb{N}$. As $\mu(\widetilde A)>0$ by construction, we conclude that $\lim\limits_{j\rightarrow\infty}\Psi(f_j)[x]=\infty=\Psi(f)[x]$, which finishes the proof.
\end{proof}

\noindent
We are now in position to complete the
\begin{proof}[Proof of Theorem~\ref{mthm:CharMonRadEquiv}]
 By Theorem~\ref{thm:MonRadCandidate}, we are left to prove that every continuous, additive, radially and $\SO(n)$-equivariant map $\Psi:\ConvO(\RR^n) \rightarrow \ConvO(\RR^n)$ is of the form~\eqref{eq:MainThmRadiallyEquiv}.
 
 First note that $\Psi$ restricts to an endomorphism of $\Conv(\RR^n,\RR)$ by Lemma~\ref{lem:restrictRadEquiv}. We can therefore consider its family of Goodey--Weil distributions $(\GW(\Psi_x))_{x \in \RR^n }$. Fixing the point $\bar e \in \unitsurfn$, every non-zero $x \in \RR^n $ is given by $x = \|x\| \vartheta_x \bar e$, and from Lemma~\ref{lem:equivPropGWdistr} we deduce that $\GW(\Psi_{\bar e})$ is $\SO(n-1)$-invariant and
 \begin{align*}
  \GW(\Psi_x)[\phi] = \GW(\Psi_{\bar e})[\phi(\|x\| \vartheta_x \cdot)], \quad \phi \in C_c^\infty(\RR^n).
 \end{align*}
 By Lemma~\ref{lem:GWdistrMonot}, there exists a family of non-negative measures $(\mu_x)_{x\in \RR^n }$ with compact support such that $\GW(\Psi_x) = \mu_x$. Hence, $\GW(\Psi_x) = \mu_x$ is given as the pushforward of the $\SO(n-1)$-invariant measure $\mu := \mu_{\bar e}$ under the map $y \mapsto \|x\| \vartheta_x y$, $x \in \RR^n  \setminus\{0\}$. Applying Lemma~\ref{lem:GWdistrMonot} once again, we deduce that
 \begin{align*}
  \Psi(f)[x] = \int_{\RR^n} f(y) d\mu_x(y) = \int_{\RR^n} f(\|x\| \vartheta_x y) d\mu(y), \quad x \in \RR^n  \setminus\{0\},
 \end{align*}
 for every $f \in \Conv(\RR^n,\RR)$. As $\Psi(f)$ is finite and therefore continuous, 
 \begin{align*}
 \Psi(f)[0] = \lim_{\|x\|\to 0} \Psi(f)[x] = \liminf_{\|x\| \to 0} \Psi(f)[x].
 \end{align*}
 Coinciding with it on the dense subset $\Conv(\RR^n,\RR)$ and being continuous, $\Psi$ must therefore be equal to the endomorphism constructed in Theorem~\ref{thm:MonRadCandidate}.
\end{proof}

 \medskip

 Each endomorphism $\Psi$ satisfying the conditions of Theorem~\ref{mthm:CharMonRadEquiv}, which is additionally dually translation-invariant, restricts to a monotone Minkowski endomorphism. Indeed, let $K \in \convexbodies$ and define $\Phi K \in \convexbodies$ by $h(\Phi K,\cdot) = \Psi(h(K,\cdot))$. Here, the radial equivariance implies that $\Psi(h(K,\cdot))$ is again $1$-homogeneous and therefore a support function, making $\Phi$ well-defined. Continuity, monotonicity, translation-invariance and $\SO(n)$-equivariance now follow from the according properties of $\Psi$.
 
 On the other hand, Theorems~\ref{thm:KiderlenMinkEndos} and \ref{thm:HSKlasseAsplEndos} show that every monotone Minkowski endomorphism can be extended to a continuous and additive map on $\ConvO(\RR^n)$. This extension is not unique. However, the extensions of Theorem~\ref{thm:HSKlasseAsplEndos} can be characterized by their action on radially symmetric functions, where a function $f:\RR^n \rightarrow \RR$ is called radially symmetric if $f(\eta x)=f(x)$ for all $x \in \RR^n$ and $\eta \in \SO(n)$. To show this, we first need the following lemma.

\begin{lemma}\label{lem:CharMultIdRadSymm}
 The map $\Psi:\ConvO(\RR^n) \rightarrow \ConvO(\RR^n)$ defined in Theorem~\ref{thm:MonRadCandidate} for some $\SO(n-1)$-invariant measure $\mu \in \MeasC^+(\RR^n)$ satisfies
 \begin{align}\label{eq:MultIdRadSymmEq}
  \Psi(f) = c \cdot f,
 \end{align}
 for every radially symmetric $f \in \ConvO(\RR^n)$ and some fixed $c > 0$, if and only if $\supp \mu \subseteq \unitsurfn$. In this case, $c = \mu(\RR^n)$.
\end{lemma}
\begin{proof}
 It is clear from \eqref{eq:exMonRadEq} that $\supp \mu \subseteq \unitsurfn$ implies \eqref{eq:MultIdRadSymmEq} for radially symmetric $f \in \ConvO(\RR^n)$ and that $c = \mu(\RR^n)$, by evaluating $\Psi(1)$.
 
 To see the converse, assume that there exists $y \in \supp \mu$ with $\|y\| \neq 1$ and let $\phi \in C_c^\infty(\RR^n)$ be a radially invariant, non-negative function such that $\phi(y)=1$, $\phi(x)=0$ for all $x \in \RR^n$ such that $|\|x\|-\|y\|| > \frac{1}{2}|\|y\|-1|$, and $\phi \geq 0$ in between. Choosing a radially symmetric $f \in \Conv(\RR^n,\RR)$ such that $\phi+f \in \Conv(\RR^n,\RR)$, we obtain
 \begin{align*}
  0 < \int_{\RR^n} \phi(x) d\mu(x) &= \GW(\Psi_{\bar e})[\phi] = \Psi(\phi+f)[\bar e] - \Psi(f)[\bar e] \\
                               &= c(\phi+f)(\bar e) - cf(\bar e) = c\phi(\bar e) = 0,
 \end{align*}
 which is a contradiction.
\end{proof}

\noindent
Theorem~\ref{mthm:CharMonRadEquiv} and Lemma~\ref{lem:CharMultIdRadSymm} now easily imply 
\begin{corollary}\label{cor:CharHSKlasseAsplEndos}
 The maps defined in Theorem~\ref{thm:HSKlasseAsplEndos} are precisely those continuous, additive, monotone, dually translation-invariant, as well as radially and $\SO(n)$-equivariant endomorphisms $\Psi$ of $\ConvO(\RR^n)$ that satisfy
 \begin{align*}
  \Psi(f) = c \cdot f,
 \end{align*}
 for every radially symmetric $f \in \ConvO(\RR^n)$ and some fixed $c > 0$.
\end{corollary}

\medskip

\noindent
In the remainder of the section, we deduce Corollary~\ref{mcor:CharRadiallyEquivWholeConv} from Theorem~\ref{mthm:CharMonRadEquiv}.

\begin{proof}[Proof of Corollary~\ref{mcor:CharRadiallyEquivWholeConv}]
 As in the proof of Theorem~\ref{mcor:CharGLEquivWholeConv}, the endomorphisms defined by \eqref{eq:MainCorRadEquivWholeConv} or by $\Psi\equiv0,\mathbbm{1}_{\{0\}}^\infty$ clearly possess all the claimed properties.
 
 Now let $\Psi$ be an endomorphism with the given properties and assume $\Psi \not \equiv 0$. Noting that we only used that the action of $\GL(n)$ is transitive on $\RR^n  \setminus \{0\}$ in the first part of the proof of Theorem~\ref{mcor:CharGLEquivWholeConv}, which is also true for the combination of the $\SO(n)$-action with radial scaling, we can conclude in the same way that $\Psi(0) = \mathbbm{1}_{\{0\}}^\infty$ or $\Psi(0) \in \Conv(\RR^n,\RR)$. Moreover, if $\Psi(0) = \mathbbm{1}_{\{0\}}^\infty$, then $\Psi \equiv \mathbbm{1}_{\{0\}}^\infty$. Indeed,  we can consider the Goodey--Weil distribution $\GW(\Psi_0)$ at $0$, which must be a multiple of the delta distribution at $0$, since in the proof of the corresponding statement of Lemma~\ref{lem:GLxInvDistr} we only used that $\{0\}$ is the unique compact orbit of the $\GL(n)$ action (which is also true in the current setting) and that the distribution is invariant under radial scalings by positive scalars.
 
 We therefore conclude that, if $\Psi$ is not constant, $\Psi$ maps $\Conv(\RR^n,\RR)$ to itself, and, by the arguments of Theorem~\ref{mthm:CharMonRadEquiv}, there exists an $\SO(n-1)$-invariant measure $\mu \in \MeasC^+(\RR^n)$, such that
 \begin{align}
  \label{eq:prfDrepPsifinite}
  \Psi(f)[x] = \int_{\RR^n} f(\|x\|\vartheta_x y) d\mu(y), \quad x \in \RR^n \setminus\{0\},
 \end{align}
 for all $f \in \Conv(\RR^n,\RR)$. It remains to prove that $\supp \mu$ consists of one point. By the $\SO(n-1)$-invariance of $\mu$, this point then must be a multiple of the fixed direction $\bar e \in \unitsurfn$, say $c \bar e$, $c \in \RR$, and $\Psi(f)[x] = \mu(\{c\bar e\})f(cx)$  for all $f \in \Conv(\RR^n,\RR)$ and by continuity for all $f \in \Conv(\RR^n)$. Noting that, if $c = 0$, there would exist $f \in \Conv(\RR^n)$ such that $\Psi(f)\equiv \infty$, we obtain $c \neq 0$, which yields the claim.
 
 In order to show that $\supp\mu$ consists of one point, we let $y \in \RR^n  \setminus \{0\}$ be arbitrary and consider $f(x) = \|x-y\|$. Then $j\cdot f$ epi-converges to $\mathbbm{1}_{\{y\}}^\infty$ and, by continuity, $j\Psi(f) = \Psi(jf) \to \Psi(\mathbbm{1}_{\{y\}}^{\infty})$, as $j \to \infty$. In fact, this convergence is pointwise due to the representation of $\Psi$ on $\Conv(\RR^n,\RR)$ in \eqref{eq:prfDrepPsifinite} and monotone convergence of the integral. Hence, $\Psi(f) \geq 0$ and $\Psi(f)[x] = 0$ if and only if $\Psi(\mathbbm{1}_{\{y\}}^{\infty})[x]=0$. As $f(w)>0$ for $w \neq y$ and $\mu$ is non-negative, $\Psi(f)[x]=0$ for $x \in \RR^n  \setminus\{0\}$ is equivalent to
 \begin{align*}
  \mu\left(\RR^n  \setminus\left\{\frac{\vartheta_x^{-1}y}{\|x\|}\right\}\right)=\mu\left(\{w \in \RR^n : \|x\|\vartheta_x w \neq y\} \right) = 0.
 \end{align*}
 As $\Psi$ is non-trivial by assumption, $\mu$ is non-zero, and we are finished if we can show that there exists $y \in \RR^n  \setminus\{0\}$ such that $\Psi(\mathbbm{1}_{\{y\}}^{\infty})[x]=0$ for some non-zero $x \in \RR^n $. Indeed, assume that this is not the case, so $\Psi(\mathbbm{1}_{\{y\}}^{\infty})=\mathbbm{1}_{\{0\}}^\infty$ for all $y \in \RR^n  \setminus\{0\}$. Then for every $f \in \Conv(\RR^n,\RR)$ and every $y \in \RR^n \setminus\{0\}$ we have
 \begin{align*}
  \Psi(f)[0] &= \Psi(f)[0] + \Psi(\mathbbm{1}_{\{y\}}^\infty)[0] = \Psi(f + \mathbbm{1}_{\{y\}}^\infty)[0] \\
             &= \Psi(f(y) + \mathbbm{1}_{\{y\}}^\infty)[0]       =  f(y)\Psi(1)[0] + \Psi(\mathbbm{1}_{\{y\}}^\infty)[0] = f(y) \mu(\RR^n),
 \end{align*}
 which, as $\mu \neq 0$, yields a contradiction by taking $f$ non-constant.
\end{proof}

\bigskip

\section{Proof of Theorem~\ref{mthm:Char1Dim}}
\label{sec:proof1Dim}

In this section, we give a proof of Theorem~\ref{mthm:Char1Dim}. We will split the proof into the ``only if''-part (Theorem~\ref{thm:1DNecesCond}) and the ``if''-part (Theorem~\ref{thm:1DSuffCond}). In the proofs, we will make use of the following characterization of continuous, dually epi-translation invariant valuations on $\RR$ by Colesanti, Ludwig and Mussnig \cite{Colesanti2019b}. Note that in \cite{Colesanti2019b}, the Hessian measure $\Theta_0(f,\cdot)$ is used instead of the Monge--Amp\`ere measure $\MongAmp{f}{\cdot}$ (see Section~\ref{subsection:examples}), which arises as a marginal of $\Theta_0$.

\begin{theorem}[\cite{Colesanti2019b}*{Cor.~6}]\label{thm:CLMCharValDim1}
	A functional $Z:\Conv(\RR,\RR)\to\RR$ is a continuous and dually epi-translation invariant valuation if and only if there exist a constant $\zeta_0 \in \RR$ and a function $\zeta_1 \in C_c(\RR)$ such that
	\begin{align*}
	Z(f) = \zeta_0 + \int_\RR \zeta_1(y) d\MongAmp{f}{y}
	\end{align*}
	for every $f \in \Conv(\RR,\RR)$. $Z$ is additive if and only if $\zeta_0 = 0$.
\end{theorem}

\noindent
We start with the ``only if''-part. In the proof, we will use the notation $y_+$ for the function $y \mapsto \max(y,0)$.
\begin{theorem}\label{thm:1DNecesCond}
 Suppose that $\Psi: \Conv(\RR,\RR) \rightarrow \Conv(\RR,\RR)$ is continuous and additive. Then the associated family of Goodey--Weil distributions of $\Psi$ is given by
 \begin{align}\label{eq:1DNecesCondRepGW}
   \GW(\Psi_x) = \partial^2_y\psi(x,\cdot), \quad x \in \RR,
 \end{align}
 where $\psi \in C(\RR^2)$ has the following properties:
 \begin{enumerate}
  \item \label{propNecesCondDim1_CondConv} $\psi(\cdot,y)$ is convex for every $y \in \RR$;
  \item \label{propNecesCondDim1_CondSupps}   For every compact subset $A\subset\RR$ there exists $R=R(A)>0$ such that
    \begin{enumerate}
        \item \label{propNecesCondDim1_CondSuppsA}$\supp\partial_x^2\psi(\cdot,y)\cap A=\emptyset$ for all $y\in\RR\setminus[-R,R]$;
        \item \label{propNecesCondDim1_CondSuppsB}$\supp\partial_y^2\psi(x,\cdot)\subseteq [-R,R]$ for all $x\in A$. %
    \end{enumerate}
 \end{enumerate}
\end{theorem}

\begin{proof}
 Let $\Psi:\Conv(\RR,\RR) \to \Conv(\RR,\RR)$ be continuous and additive and denote by $u_x = \GW(\Psi_x)$ the associated family of Goodey--Weil distributions. Define the map
 \begin{align*}
  \widetilde \Psi(f)[x] = \Psi(f)[x] - f(0)\cdot\Psi(y \mapsto 1)[x] - (f(1)-f(0))\cdot\Psi(y \mapsto y)[x],
 \end{align*}
 for $f \in \Conv(\RR,\RR)$. Then, for every $x \in \RR$, the map $f \mapsto \widetilde\Psi(f)[x]$ is a continuous, additive and dually epi-translation invariant map (hence, a valuation). By Theorem~\ref{thm:CLMCharValDim1} and Theorem~\ref{thm:defMongAmp}, there exists a function $\zeta(x,\cdot) \in C_c(\RR)$ such that
 \begin{align*}
  \widetilde \Psi(f)[x] = \int_\RR \zeta(x,y) f''(y) dy, \quad f \in \Conv(\RR,\RR) \cap C^2(\RR),
 \end{align*}
 and we conclude that the Goodey--Weil distributions $u_x$ of $\Psi$ are given by
 \begin{align*}
  u_x = \Psi(y \mapsto 1)[x]\cdot \delta_0 + \Psi(y \mapsto y)[x] \cdot (\delta_1 - \delta_0) + \partial^2_y \zeta(x,\cdot), \quad x \in \RR.
 \end{align*}
 Convoluting the right-hand side with the fundamental solution $\varphi(y) = y_+$ of $\partial^2_y$, we have shown \eqref{eq:1DNecesCondRepGW} with
 \begin{align*}
  \psi(x,y) = \Psi(y \mapsto 1)[x]\cdot y_+ + \Psi(y \mapsto y)[x] \cdot ((y-1)_+ - y_+) + \zeta(x,y), \quad x,y \in \RR.
 \end{align*}
 Note that $\psi$ is clearly continuous in $y$ and that, by Lemma~\ref{lem:AffFctsMapToAff}, the first two summands are affine in $x$. Moreover, as $\psi(x,y) = (u_x \ast \varphi)(y)$, we can approximate $\varphi$ with smooth and convex functions $\varphi_\varepsilon$ to obtain
 \begin{align}\label{eq:ProofNecesCond1DCalcPsi}
  \psi(x,y) = \lim_{\varepsilon \to 0}(u_x \ast \varphi_\varepsilon)(y) = \lim_{\varepsilon \to 0} \Psi(\varphi_\varepsilon(y-\cdot))[x] =  \Psi(\varphi(y-\cdot))[x],
 \end{align}
 so $\psi(\cdot,y)$ is in particular convex for every $y\in\RR$, which implies property~\eqref{propNecesCondDim1_CondConv}. Moreover, the continuity of $\Psi$ implies that $\psi(\cdot,y)$ depends continuously on $y\in\RR$. Lemma~\ref{lem:epiConvEquiv}(\ref{lem:epiConvEquivUnifComp}) and \eqref{eq:ProofNecesCond1DCalcPsi} thus imply that $\psi \in C(\RR^2)$.
 
 Now let $A\subset\RR$ be a compact subset. Property~\eqref{propNecesCondDim1_CondSuppsB} 
 then follows from Proposition~\ref{propLocalBoundSupport}, if we choose $R = R(A)$ appropriately. In order to show property~\eqref{propNecesCondDim1_CondSuppsA}, note that the bound $\supp u_x \subseteq [-R,R]$ implies that $\Psi(\varphi(y-\cdot))[x]$ is affine for $y>R$ and zero for $y<-R$, for every $x \in A$. Indeed, if $y>R$, $\varphi(y-\cdot)$ is affine on $\supp u_x$, and if $y<-R$, $\varphi(y-\cdot)$ is zero on $\supp u_x$. Consequently, using also \eqref{eq:ProofNecesCond1DCalcPsi}, the distributions $\partial^2_x \psi(\cdot, y)$, $y \in \RR \setminus [-R,R]$, coincide on $A$ with the second derivative of an affine function and therefore vanish. This shows property~\eqref{propNecesCondDim1_CondSuppsA} and completes the proof.
\end{proof}

\noindent The second part of Theorem~\ref{mthm:Char1Dim} is the content of the following theorem.

\begin{theorem}\label{thm:1DSuffCond}
 Suppose that $\psi \in C(\RR^2)$ satisfies
 \begin{enumerate}
  \item \label{propSuffCondDim1_CondConv} $\psi(\cdot,y)$ is convex for every $y \in \RR$.
  \item \label{propSuffCondDim1_CondSupps}   For every compact subset $A\subset\RR$ there exists $R=R(A)>0$ such that
    \begin{enumerate}
        \item \label{propSuffCondDim1_CondSuppsA}$\supp\partial_x^2\psi(\cdot,y)\cap A=\emptyset$ for all $y\in\RR\setminus[-R,R]$;
        \item \label{propSuffCondDim1_CondSuppsB}$\supp\partial_y^2\psi(x,\cdot)\subseteq [-R,R]$ for all $x\in A$. %
    \end{enumerate}
 \end{enumerate}
 Then $u_x = \partial^2_y\psi(x,\cdot)$, $x \in \RR,$ defines a family of Goodey--Weil distributions of a continuous and additive map $\Psi:\Conv(\RR,\RR) \rightarrow \Conv(\RR, \RR)$.
\end{theorem}
\begin{proof}
 Let $A\subset\RR$ be a compact and convex subset, $x\in A$ and set $R:=R(A)$. By property~\eqref{propSuffCondDim1_CondSuppsB}, the function $y \mapsto \psi(x,y)$ is affine on $(-\infty,-R]$ and $[R,\infty)$, so there exist $c_1(x) , c_2(x) , c_3(x) , c_4(x) \in \RR$ such that
 \begin{align*}
 	\psi(x,y)=& c_1(x) y+c_2(x)(y+1)\quad\text{for }y\ge R\\
 	\psi(x,y)=& c_3(x) \cdot (-y)+c_4(x)(-y-1)\quad\text{for }y\le -R.
 \end{align*}
 Plugging in $R, R+1$ in the first equation and $-R,-R-1$ in the second, we obtain
 \begin{align*}
 	c_1(x)=& (R+1)\psi(x,R+1) - (R+2)\psi(x,R),\\%
 	c_2(x)=& (R+1)\psi(x,R) - R\psi(x,R+1),\\%
 	c_3(x)=& (R+1)\psi(x,-R-1) - (R+2)\psi(x,-R),\\%
 	c_4(x)=& (R+1)\psi(x,-R) - R\psi(x,-R-1). %
 \end{align*}
 Then the map
 \begin{align*}
  \widetilde \psi(x,y) = \psi(x,y) - \left(c_1(x) y_+ + c_2(x) (y+1)_+ + c_3(x) (-y)_+ + c_4(x) (-y-1)_+\right)
 \end{align*}
 vanishes for $y \in \RR \setminus [-R,R]$ and $x\in A$. As $\psi$ is continuous, the coefficients $c_i(x), i\in\{1,\dots,4\},$ depend continuously on $x\in A$. Moreover, we deduce from property~\eqref{propSuffCondDim1_CondSuppsA} that $\psi(\cdot,y)$ is affine on $A$ for $y\in\RR\setminus[-R,R]$. By continuity, this holds for $y=\pm R$ as well, so in particular, $c_i(x)$ defines an affine function on $A$. Together with property~\eqref{propSuffCondDim1_CondConv}, this implies that $x\mapsto \widetilde \psi(x,y)$ is convex on $A$ for every $y \in \RR$.
  
 Applying the ``if''-part of Theorem~\ref{thm:CLMCharValDim1}, $\widetilde\psi(x,\cdot)$ defines a dually epi-translation invariant, continuous valuation on $\Conv(\RR,\RR)$ by
 \begin{align*}
  f \mapsto \int_{\RR}\widetilde\psi(x,y) d\MongAmp{f}{y}, \quad f \in \Conv(\RR,\RR).
 \end{align*}
 For $x\in A$ we define $\Psi_x:\Conv(\RR,\RR)\rightarrow\RR$ by
 \begin{align*}
  \Psi_x(f) := (c_1(x)-c_3(x))f(0) + (c_2(x)-c_4(x))f(-1)+  \int_{\RR}\widetilde\psi(x,y) d\MongAmp{f}{y}.
 \end{align*}
 Then $\Psi_x$ is continuous and additive. Moreover, it is easy to see that
 \begin{align*}
 	\GW(\Psi_x)[\phi]=\partial_y^2\psi(x,\cdot) = u_x(\phi),\quad \phi\in C^\infty_c(\RR),
 \end{align*}
 so $\Psi_x$ is uniquely determined by the distribution $u_x$ and in particular, $\Psi_x$ is independent of the choice of $A$. For every $x\in \RR$ we thus obtain a unique continuous and additive map $\Psi_x:\Conv(\RR,\RR)\rightarrow\RR$ extending the distribution $u_x$. Setting $\Psi(f)[x]:=\Psi_x(f)$, we obtain a real valued function $\Psi(f)$ for every $f\in\Conv(\RR,\RR)$.
 
 \bigskip
 
 It remains to see that the function $\Psi(f)$ is convex for all $f\in\Conv(\RR,\RR)$. Note that it is sufficient to show that this function is convex on any compact and convex subset $A\subset\RR$. But by the previous discussion, $\Psi(f)[x]$ is given for any $x\in A$ by
 \begin{align*}
 	\Psi(f)[x] = (c_1(x)-c_3(x))f(0) + (c_2(x)-c_4(x))f(-1)+  \int_{\RR}\widetilde\psi(x,y) d\MongAmp{f}{y},
 \end{align*} 
 where $c_1(x) , c_2(x) , c_3(x) , c_4(x)$ are affine functions on $A$ and $x\mapsto \widetilde\psi(x,y)$ is convex on $A$ for every $y \in \RR$. As the Monge--Amp\`ere measure of a convex function is non-negative, the right-hand side of this equation defines a convex function on $A$.
 
 \bigskip
 
 The prescription $f\mapsto \Psi(f)$ therefore defines an additive map $\Psi:\Conv(\RR,\RR)\rightarrow\Conv(\RR,\RR)$. As $f\mapsto \Psi(f)[x]$ is continuous for every $x\in \RR$, $\Psi$ is continuous by Lemma~\ref{lem:epiConvEquiv}(\ref{lem:epiConvEquivPWonDense}).
\end{proof}

\noindent
The proof of Theorem~\ref{mthm:Char1Dim} is completed by

\begin{proposition}
 Suppose that $\psi \in C(\RR^2)$ satisfies the conditions of Theorem~\ref{thm:1DSuffCond}.
 
 \noindent
 Then the map $\Psi: \Conv(\RR^n) \rightarrow \Conv(\RR^n)$, defined by Theorem~\ref{thm:1DSuffCond}, is monotone if and only if $\psi(x,\cdot)$ is convex for every $x \in \RR$.
\end{proposition}
\begin{proof}
 By Theorem~\ref{thm:1DSuffCond}, the Goodey--Weil distributions of $\Psi$ are given by $\partial_y^2 \psi(x,\cdot)$, and, by Lemma~\ref{lem:GWdistrMonot}, $\Psi$ is monotone if and only if its Goodey--Weil distributions are given by non-negative measures. Hence, $\Psi$ is monotone if and only if $\partial_y^2\psi(x,\cdot)$ is non-negative, which is equivalent to $\psi(x,\cdot)$ being convex, by \cite{Hoermander2003}*{Thm.~4.1.6}.
\end{proof}

\medskip

\noindent
Let us also add the following note concerning the uniqueness of $\psi$.
\begin{corollary}
	\label{cor:Uniqueness1Dpsi}
	Let $\psi,\tilde{\psi}\in C(\RR^n)$ satisfy the properties in Theorem \ref{thm:1DSuffCond}. Then the induced endomorphisms of $\Conv(\RR,\RR)$ coincide if and only if $\psi(x,\cdot)-\tilde{\psi}(x,\cdot)$ is affine for all $x\in\RR$.
\end{corollary}
\begin{proof}
	$\psi$ and $\tilde{\psi}$ define the same endomorphism if and only if $\partial_y^2\psi(x,\cdot)=\partial_y^2\psi(x,\cdot)$ as distributions, that is, if $\partial_y^2(\psi(x,\cdot)-\tilde{\psi}(cx,\cdot))=0$. This is the case if and only if $\psi(x,\cdot)-\tilde{\psi}(cx,\cdot)$ is an affine function.
\end{proof}

\noindent 
We complete the section with a continuation of Example~\ref{ex:IntPhiT} in view of Theorem~\ref{mthm:Char1Dim}.

\begin{example}
 For the endomorphism $\Psi_\varphi$ from Example~\ref{ex:IntPhiT}, $\varphi \in \Conv(\RR,\RR)$, the function $\psi$ from Theorem~\ref{thm:1DNecesCond} is given up to addition of affine functions by
 \begin{align*}
  \psi(t,s) = \begin{cases}
               0, & \text{for } |s| \geq \varphi(t),\\
               \frac{(s+\varphi(t))^2}{2} - 2\varphi(t) s_+, & \text{for } |s| < \varphi(t).
              \end{cases}
 \end{align*}
\end{example}

\bibliographystyle{abbrv}

\bibliography{./ArtikelAsplEndomorphismen_v20220719_arxivSubmission}

\end{document}